\documentclass{amsart}
\usepackage{newlfont}
\usepackage{graphicx}
\usepackage{indentfirst}
\usepackage{fancyhdr}
\usepackage{amsmath}
\usepackage{amsthm}
\usepackage{latexsym}
\usepackage[psamsfonts]{amssymb}
\usepackage{mathrsfs}
\usepackage{frcursive}

\newtheorem{theorem}{Theorem}
\newtheorem{proposition}[theorem]{Proposition}
\newtheorem{lemma}[theorem]{Lemma}
\newtheorem{corollary}[theorem]{Corollary}
\newtheorem{definition}[theorem]{Definition}

\newtheorem{remark}[theorem]{Remark}
\renewcommand{\thetheorem}{\thesection.\arabic{theorem}}
\renewcommand{\theproposition}{\thesection.\arabic{proposition}}
\renewcommand{\thelemma}{\thesection.\arabic{lemma}}
\renewcommand{\thedefinition}{\thesection.\arabic{definition}}

\renewcommand{\theequation}{\thesection.\arabic{equation}}
\renewcommand{\theremark}{\thesection.\arabic{remark}}
\DeclareMathOperator{\sign}{sign}
\DeclareMathOperator{\supp}{supp}

\allowdisplaybreaks

\begin{document}

\title[On a generalization of the M\'etivier operator]{On the sharp Gevrey regularity for a generalization of the M\'etivier operator}


\author{G. CHINNI}
\address{Department of Mathematics, University of Bologna, Piazza di
  Porta S. Donato 5, 40127 Bologna, Italy}
\email{gregorio.chinni@gmail.com}
\subjclass[2010]{35H10, 35H20, 35B65 }
\date{\today}
%
\begin{abstract}
  We prove a sharp Gevrey hypoellipticity for the operator 
	\begin{align*}
	D_{x}^{2}+\left(x^{2n+1}D_{y}\right)^{2}+\left(x^{n}y^{m}D_{y}\right)^{2},
	\end{align*}
in $\Omega$ open neighborhood of the origin in $\mathbb{R}^{2}$, where
$n$ and $m$ are positive integers. The operator is a non trivial
generalization of the M\'etivier operator studied in
\cite{M81}. However it has a symplectic characteristic manifold and a
non symplectic stratum according to the Poisson-Treves
stratification. According to Treves conjecture it turns out not to be
analytic hypoelliptic.  
\end{abstract}

\maketitle
\tableofcontents
\section{Introduction}
\renewcommand{\theequation}{\thesection.\arabic{equation}}
\setcounter{equation}{0}
\setcounter{theorem}{0}
\noindent
In \cite{M81} G. M\'etivier studied the non-analytic hypoellipticity
for a class of second order partial differential operators with
analytic coefficients  
in $\Omega$, open neighborhood of the origin in $\mathbb{R}^{n}$,
whose principal symbol vanishes exactly of order two on a submanifold
of 
$T^{*}\Omega$.

In the case of sum of squares of vector fields the most representative
model of such class is the following  
\begin{equation}
  \label{Met-Op}
D_{x}^{2}+ x^{2}D_{y}^{2} + \left(y D_{y}\right)^{2}.
\end{equation}
In \cite{M81} the author proves that the operator (\ref{Met-Op}) is
$G^{2}$-hypoelliptic and not better in a neighborhood of the origin.

This means that if $ u $ is a smooth function in the neighborhood of
the origin $ U $ solving there 
the equation $ Pu = f $, where $ P $ denotes the operator in
\eqref{Met-Op} and $ f $ is analytic in $ U $, then $ u $ belongs to the Gevrey
class $ G^{2}(U) $ and there are no solutions in $ G^{s}(U) $ with $ 1
\leq s < 2$.

Here $G^{s}(U)$, $s\geq 1$,  denotes the space of Gevrey functions in
$U$, i.e. the set of all $f \in C^{\infty}(U)$ 
such that for every compact set $K \subset U$ there are two positive
constants $C_{K}$ and $A$ such that for every $\alpha \in
\mathbb{Z}^{n}_{+}$
$$
|D^{\alpha}f(x)| \leq A C^{|\alpha|}_{K} |\alpha|^{s|\alpha|}, \qquad \forall x\in K.
$$

The Gevrey index 2 corresponds to that obtained by Derridj and
Zuily, \cite{DZ} (see also \cite{ABC}), when applied to this operator.

For this type of operators H\"ormander, \cite{H67}, stated the
condition for $ C^{\infty} $ hypoellipticity:
\begin{itemize}
  \item[\textbf{(H)}]{}
the Lie algebra generated by the vector fields and their commutators
has dimension equal to the dimension of the ambient space.
\end{itemize}
We recall that a sum of squares operator satisfying the H\"ormander
condition is said to be $G^{s}$-hypoelliptic, $s\geq 1$, 
in $\Omega$, open subset of $\mathbb{R}^{n}$, if for any $U$ open
subset of $\Omega$ the conditions 
$u\in \mathscr{D}'\left(U\right)$ and $Pu \in G^{s}(U)$ imply that $u \in G^{s}(U)$. 

Needless to say $G^{1}(U) = C^{\omega}(U)$, the class of real-analytic functions on $U$.

\bigskip

One of the main motivations to study the optimality of the Gevrey
regularity of solutions to sums of squares in two variables is due to
the fact that in dimension greater or equal than 4 the Treves
conjecture has been proved to be false, see \cite{ABM-16},
\cite{BM-16}. In dimension 3 there are no results although in
\cite{bm-surv} a candidate has been produced that should violate the
conjecture. On the contrary there are good reasons to surmise that the
conjecture of Treves holds in two variables. We refer to \cite{Tr-99},
\cite{BoTr-04} for more details on the statement of the conjecture, as
well as to \cite{bm-surv} for a discussion of both the 3- and
2-dimensional cases.

This implies that it is of crucial importance to know if a certain
Gevrey regularity, which may be relatively easy to obtain by using $
L^{2} $ a priori estimates, is optimal or not.

In two variables this becomes particularly difficult because if the
characteristic set is not a symplectic real analytic manifold, the
Hamilton leaves corresponding to the kernel of the symplectic form
have injective projection onto the fibers of the cotangent bundle,
thus causing great technical complication in the construction of a
singular solution, which is the method of proving optimality.

In order to exhibit a singular solution, meaning a solution which is
not more regular than predicted, one constructs first a so called
asymptotic formal solution. Once this is done, the formal solution
gives a true solution solving the same equation with possibly a
different right hand side and then one can finally achieve the proof. 

The construction of the formal solution in \cite{M81}, e.g. for the
operator in \eqref{Met-Op}, uses the spectral theory of the harmonic
oscillator, i.e. of the operator $ D_{x}^{2} + x^{2} $, the variable $
y $ being reduced to a parameter. For the eigenfunctions of the
harmonic oscillator there are three terms recurrence formulas relating
the derivative of an eigenfunction to those up and down one notch.
As a result the $ k $-th derivative with respect to $ x $ of an
eigenfunction can be expressed as a linear combination of $ 2k $
eigenfunctions and this makes possible turning the formal solution
into a true one.

In the case of an anharmonic oscillator, which occurs if one considers
cases vanishing of order higher than 2, Gundersen has shown in
\cite{Gundersen_76} that such recurrence formulas do not exist, so
that the optimality for the operators
$$ 
D_{x}^{2} + x^{2(q-1)} D_{y}^{2} + (y^{k} D_{y})^{2} ,
$$
as well as
$$ 
D_{x}^{2} + x^{2(q-1)} D_{y}^{2} + x^{2(p-1)} (y^{k} D_{y})^{2} ,
\quad 1 <p < q,  
$$
is not known.

In 2001 Bender and Wang, \cite{Bender2001}, studied the following
class of eigenvalue problems 
\begin{align}\label{BW_EigPr}
-u''(t) + t^{2N+2} u(t)= t^{N} E \ u(t), \qquad N=-1,0,1,2,\dots,
\end{align}
on the interval $-\infty <t < +\infty$. Such a kind of problem arises
in different contexts in physics, like the fluid-flow with a resonant
internal boundary layer and in supersymmetric quantum mechanics. The
eigenfunction $ u $ above is a confluent hypergeometric function and
moreover can be written as a product of a polynomial and a function
exponentially vanishing at infinity.

In this paper we study the optimality of the Gevrey regularity for the
operator 
\begin{equation}
  \label{Op-GM}
M_{n,m}(x,y;D_{x},D_{y}) =
D_{x}^{2}+\left(x^{2n+1}D_{y}\right)^{2}+\left(x^{n}y^{m}D_{y}\right)^{2},
\end{equation}
in $\Omega$ open neighborhood of the origin in $\mathbb{R}^{2}$. Here
$n$ and $m$ are positive integers. 
We point out that $M_{n,m}$ is a generalization of the M\'etivier
operator, \eqref{Met-Op}, which corresponds to the case $n=0$ and $m=1$.

This operator allows us to use the recurrence relations between the
eigenfunctions of the operator in order to construct an asymptotic
solution. 

Here is the statement of the result.
\begin{theorem}\label{Th}
The operator $M_{n,m}$, (\ref{Op-GM}), is
$G^{\frac{2m}{2m-1}}$-hypoelliptic and not better in any neighborhood
of the origin.   
\end{theorem}
A few remarks are in order.
\begin{itemize}
	\item[(a)] The Gevrey regularity obtained for the operator
          $M_{n,m}$ is in accordance with that predicted in
          \cite{BT2014}.
	\item [(b)] The characterization in terms of
          \textit{Gel'fand-Shilov spaces}, see Theorem \ref{G-S_Char}
          below, of the eigenfunctions of the eigenvalue problem
          (\ref{Eig_P}) allows us to precisely compute the partial
          Gevrey regularity for $M_{n,m}$. This means that we find the
          Gevrey regularity with respect to $ x $ and with respect to
          $ y $ of the solutions: 
	  the operator (\ref{Op-GM}) is Gevrey hypoelliptc of order
          $s_{0}= 1+ \frac{1}{(2m-1)(2n+2)}$ 
	  with respect to the variable $x$ and of order $s_{1}=
          \frac{2m}{2m-1}$ with respect to the variable $y$ 
	and not better in any neighborhood of the origin.
	We recall that a smooth function $u(x,y)$ belongs to the
        non-isotropic Gevrey space $G^{(s_{0},s_{1})}(U)$, 
	$U$ open subset of $\mathbb{R}^{2}$, if for every compact set
        $K \subset U$ there are two positive constants $C_{K}$ and $A$
        such that for every $\alpha, \, \beta \in \mathbb{Z}_{+}$
$$
|D^{\alpha}_{x}D^{\beta}_{y} u(x,y)| \leq A C^{\alpha +\beta}_{K}
\alpha^{s_{0}\alpha}\beta^{s_{1}\beta}, \qquad  (x,y)\in K. 
$$
	\item[(c)] As well as the M\'etivier operator (\ref{Met-Op}),
          the operator $M_{n,m}$, in (\ref{Op-GM}), is not globally
          analytic hypoelliptic on the two dimensional torus. This
          means that if for instance we consider
          in $\mathbb{T}^{2}$ the operator 
$$
D_{x}^{2}+ \left(\sin^{2n+1}(x)D_{y}\right)^{2}+
\left(\sin^{n}(x)\sin^{m}(y)D_{y}\right)^{2} ,
$$
	then it is not globally analytic hypoelliptic in
        $\mathbb{T}^{2}$, it is $G^{s}$-globally hypoelliptic
	for every $s\geq \frac{2m}{2m-1}$. This can be obtained
       via Theorem \ref{Th} and or following the proof of the Theorem 2.1 in  \cite{Tr-2006} (p. 325)
       or following the proof of the Proposition 1.1 in \cite{BC-2022}, concerning the subject see also \cite{chinni2}. 
       We recall that a partial differential
        operator $P$ is said to be $G^{s}$-globally hupoelliptic,
        $s\geq 1$, in $\mathbb{T}^{n}$ if the condition $u\in
        \mathscr{D}'\left(\mathbb{T}^{n}\right)$ and $Pu \in
        G^{s}(\mathbb{T}^{n})$ imply that $u \in
        G^{s}(\mathbb{T}^{n})$. 
	Here $G^{s}(\mathbb{T}^{n})$ denotes the space of the Gevrey
        functions on $\mathbb{T}^{n}$.
\end{itemize}
Here is the plan of the paper.
In the first section we construct a formal solution to the
problem $M_{n,m}u=0$. More precisely first of all we construct a formal
solution and subsequently,
in order to justify the formal steps we went through, we turn the
formal solution, $ u $, into a true solution, $ \tilde{u} $, with the
aid of a family of smooth cutoff functions. 
Then $M_{n,m}\tilde{u}$ belongs to a suitable function space.

In the second section we deduce the proof of the Theorem \ref{Th}.
In order to focus on the construction of the formal solution, we
shifted to the appendix a number of both known and unknown facts about
the properties of the eigenfunctions of the eigenvalue problem
(\ref{BW_EigPr}), in the case $N=2n+1$, as well as some auxiliary
technical lemmas.

\noindent
\textbf{Aknowledgements.} The preparation of this manuscript has been
done while the author was supported by the Austrian Science Fund
(FWF), Lise-Meitner position, project no. M2324-N35. The author would
like to thank the people of the Department of Mathematics of Vienna
University and in particular professor B. Lamel, for their hospitality and
the numerous mathematical discussions. Moreover I would like to thank
professor A. Bove for a number of very useful an fruitful discussions and 
for his comments on various aspects of this work.

%
%
\section{Formal solution of the operator \eqref{Op-GM}
} 
\renewcommand{\theequation}{\thesection.\arabic{equation}}
\setcounter{equation}{0} \setcounter{theorem}{0}
\noindent
Following the ideas in \cite{M81}, the purpose of the present section is to construct a formal and the associate ``approximate" solution to the problem
$M_{n,m}u=0$ of the form
\begin{align}\label{Sol_0}
\mathscr{K}\!\left[u\right]\!(x,y)\!= \!\int_{0}^{+\infty} \!\!\!e^{i\rho^{\theta}y} \rho^{r} u\left(\rho^{\gamma}x,\rho\right) d\rho
= \int_{0}^{+\infty} \!\!\!e^{i\rho^{\theta}y} \rho^{r} \left[u\left(t,\rho\right)\right]_{|_{t=\rho^{\gamma}x}}\!\! d\rho,
\end{align}
where $\theta$, $\gamma$ and $r$ are parameters that will be chosen later and the function $u(t,\rho)$
is an infinitely differentiable function in $\mathbb{R}^{2}_{t,\rho} $ with support in the region $\rho>0$ and
rapidly decreasing as $\rho $ goes to infinity. 
The first step will be to establish the values of the parameters $\theta$ and $\gamma$.\\
We have
\begin{enumerate}
	\item $D_{x}^{2}\mathscr{K}\left[u\right]
	=\displaystyle\int_{0}^{+\infty} \!\! e^{i\rho^{\theta}y} \rho^{r} \rho^{2\gamma}\left[D_{t}^{2}u(t,\rho)\right]_{|_{t=\rho^{\gamma}x}} \, d\rho$;
	\item $\left(x^{2n+1}D_{y}\right)^{2}\mathscr{K}\left[u\right]
	=\displaystyle\int_{0}^{+\infty} \!\! e^{i\rho^{\theta}y} \rho^{r} \rho^{2\theta-2(2n+1)\gamma}\left[t^{2(2n+1)}u(t,\rho)\right]_{|_{t=\rho^{\gamma}x}}\!\!\! d\rho$;
	\item and by applying the Lemma \ref{Tech_L3}, see Appendix, 
	\begin{align*}
	&\left(x^{n}y^{m}D_{y}\right)^{2}\mathscr{K}\left[u\right]= \left(x^{2n}y^{2m}D_{y}^{2}+ \frac{m}{i} x^{2n}y^{2m-1}D_{y}\right) \mathscr{K}\left[u\right]
	\\
	&=\left(\frac{1}{i\theta}\right)^{2m}\!\!\!\int_{0}^{+\infty} \!\!\!\! e^{i\rho^{\theta}y} \rho^{r+2\theta-2n\gamma+2m(1-\theta)}\left[t^{2n}
	\mathcal{P}(t,\rho,\partial_{t},\partial_{\rho})u(t,\rho)\right]_{|_{t=\rho^{\gamma}x}} \!\!\!d\rho,
	\end{align*}
	where
	\begin{align*}
	\mathcal{P}(t,\rho,\partial_{t}, \partial_{\rho})=
	\partial_{\rho}^{2m} &+ \frac{2m}{\rho}\!\left(\frac{2m+1}{2}-\theta(m-1)+r+ \gamma t\partial_{t}\right)\!\partial_{\rho}^{2m-1}\\
	&+\sum_{i=2}^{2m}\frac{1}{\rho^{i}}\left(\mathscr{P}_{i}^{1}(t\partial_{t})-m\theta\mathscr{P}_{i-1}^{2}(t\partial_{t})\right)\partial_{\rho}^{2m-i},
	\end{align*}
	here
	\begin{align*}
	\mathscr{P}_{i}^{1}(t\partial_{t})=\sum_{j=0}^{i} \text{\textcursive{p}}_{i,j}^{1}(t\partial_{t})^{j}
	\text{ and }  \mathscr{P}_{i-1}^{2}(t\partial_{t})=\sum_{j=0}^{i-1} \text{\textcursive{ p}}_{i-1,j}^{2}(t\partial_{t})^{j}.
	\end{align*}
	The coefficients $\text{\textcursive{ p}}_{i,j}^{1}$ and $\text{\textcursive{ p}}_{i,j}^{2}$ are obtained by the formulas
	(\ref{Coef_1}) and (\ref{Coef_2}), Lemma \ref{Tech_L3}, setting $p=2m$, $q=r+2\theta-2n\gamma$, $f=2n$ and  $p=2m-1$, $q=r+\theta-2n\gamma$, $f=2n$ respectively.
\end{enumerate}
Choosing $\theta=2m(2m-1)^{-1}$ and $\gamma=m(n+1)^{-1}(2m-1)^{-1}$ we obtain
\begin{align}\label{Act_M-K}
& M_{n,m}\mathscr{K}[u]= \left(
                               D_{x}^{2}+\left(x^{2n+1}D_{y}\right)^{2}+\left(x^{n}y^{m}D_{y}\right)^{2}
                               \right) \mathscr{K}\!\left[u\right](x,y)
\\
\nonumber
&=\int_{0}^{+\infty} \!\! e^{iy\rho^{\frac{2m}{2m-1}}} \rho^{r+\frac{2m}{(n+1)(2m-1)}}\left[\sum_{i=0}^{2m}\frac{1}{\rho^{i}}\mathcal{P}_{i}(t,\partial_{t},\partial_{\rho})u(t,\rho)\right]_{\vline_{t=\rho^{\frac{m}{(n+1)(2m-1)}}x}} \!\!\!\!\!\!\!\! d\rho,
\end{align}
where
\begin{enumerate}
	\item[i)] $\mathcal{P}_{0}(t,\partial_{t},\partial_{\rho})=-\partial_{t}^{2}+t^{2(2n+1)}+\left(\frac{2m-1}{2im}\right)^{2m}t^{2n}\partial_{\rho}^{2m}$;
	\item[ii)] $\mathcal{P}_{1}(t,\partial_{t},\partial_{\rho})= t^{2n}\left(\frac{2m-1}{2im}\right)^{2m} \left( \frac{(4m-1)2m}{2m-1}+2mr+\frac{2m^{2}}{(n+1)(2m-1)}t\partial_{t}\right)\partial_{\rho}^{2m-1}$;
	\item[iii)] $\mathcal{P}_{i}(t,\partial_{t},\partial_{\rho})= t^{2n}\left(\frac{2m-1}{2im}\right)^{2m}\mathscr{P}_{i}(t\partial_{t})\partial_{\rho}^{2m-i}$,
	$i= 2, 3,\dots,2m$, where
	$\mathscr{P}_{i}(t\partial_{t})=\sum_{j=0}^{i}\text{\textcursive{ p}}_{i,j}\left(t\partial_{t}\right)^{j}$;
	the coefficients $\text{\textcursive{ p}}_{i,j}$ are obtained from the previous formulas replacing $\theta$ and $\gamma$ with the above assigned values.
\end{enumerate}
We have to solve the following equation
\begin{align}
  \label{Pu=0}
\sum_{i=0}^{2m}\frac{1}{\rho^{i}}\mathcal{P}_{i}(t,\partial_{t},\partial_{\rho})u(t,\rho)=0.
\end{align}
We do this formally. We set
\begin{align*}
u(t,\rho)= \sum_{\ell \geq 0} u_{\ell}(t,\rho),
\end{align*}
with the purpose to obtain the functions $u_{\ell}(t,\rho)$ recursively taking advantage of
the eigenvalue problem, (\ref{Eig_P}), studied in the Appendix. More
precisely we want to express the functions $u_{\ell}(t,\rho)$
in the following form:
\begin{align*}
u_{\ell}(t,\rho)= \sum_{p=0}^{\ell} g_{\ell,p}(\rho)v_{p}(t),
\end{align*}
where $v_{p}(t)$ are the eigenfunctions given by (\ref{Ev_Eigf}) in
the Appendix.

We remark that we allow the above sum for $ u_{\ell} $ to be finite
because of the relation (\ref{Rec_Lag_1}), Lemma \ref{Rk-Ric-Rel}.
We point out that the relation (\ref{Rec_Lag_1}), and more generally the relation (\ref{Rec_Lag_2}), allow us to construct a suitable system
in order to obtain recursively the functions $u_{\ell}(t,\rho)$. We stress that $\mathcal{P}_{i}u_{\ell-i} $ can be expressed by a linear combination
of $v_{0}(t), \dots, v_{\ell}(t)$. This gives us the possibility, at
any step of the process, to reduce the problem to solving a system of
ordinary differential equations.

We choose
\begin{align}
  \label{u0}
u_{0}(t,\rho) = g_{0,0}(\rho)v_{0}(t)= e^{-c_{1}\rho}v_{0}(t),
\end{align}
where $c_{1}=\frac{2m}{2m-1}(2n+1)^{1/2m}\left(\sin\left(\frac{\pi}{2m}\right)-i \cos\left(\frac{\pi}{2m}\right)\right)$.
We have
\begin{align}\label{Eq_0}
\mathcal{P}_{0}(t,\partial_{t},\partial_{\rho})u_{0}(t,\rho)=0
\end{align}
We set $c_{0}= \Re(c_{1})$.

Before constructing our formal solution we point out that, in order to have the desired growth of the functions $g_{\ell,p}(\rho)$,
since they have exponential nature, the derivative with respect to the parameter $\rho$ does not help to push down the order with respect to negative power of $\rho$.
The only instrument that allow to do it, it is the multiplication by
negative powers of $\rho$. In particular the growth of $
g_{\ell,0}(\rho)$ will be the most thorny; the precise choice of the
parameter $r$ plays a fundamental role in this situation.

Let us start with the action of the operators $\mathcal{P}_{i}(t,\partial_{t},\partial_{\rho})$ on $u_{\ell}(t,\rho)$:
\begin{enumerate}
	\item[i)] case of $\mathcal{P}_{0}$; for every $\ell \geq 1$ we have
          \begin{align*}
        \mathcal{P}_{0} u_{\ell} 
	&= \left(-\partial_{t}^{2} + t^{2(2n+2)} + \left(\frac{2m-1}{2mi}\right)^{2m} t^{2n} \partial_{\rho}^{2m}\right)
	\left(\sum_{p=0}^{\ell} g_{\ell,p}(\rho) v_{p}(t)\right)\\
	&= t^{2n} \left(\frac{2m-1}{2mi}\right)^{2m} \sum_{p=0}^{\ell} \left( \partial_{\rho}^{2m} + \left(\frac{2mi}{2m-1}\right)^{2m} E_{p}\right)g_{\ell,p}(\rho) v_{p}(t)\\
	& \doteq t^{2n} \left(\frac{2m-1}{2mi}\right)^{2m} \sum_{p=0}^{\ell}\Theta_{p} g_{\ell,p}(\rho) v_{p}(t);
	\end{align*}
        here $ \Theta_{p} $ is defined in \eqref{ODE}.
	\item [ii)] case of  $\mathcal{P}_{1}$; for every $\ell \geq 2$ we have
          \begin{align*}
            \mathcal{P}_{1}u_{\ell - 1} 
	& = t^{2n} \left(\frac{2m-1}{2mi}\right)^{2m} \left(  \text{\textcursive{ p}}_{1,0} + \text{\textcursive{ p}}_{1,1} t\partial_{t} \right)\partial_{\rho}^{2m-1}
	\left(\sum_{p=0}^{\ell-1} g_{\ell,p}(\rho) v_{p}(t)\right)\\
	\nonumber
	&=t^{2n} \left(\frac{2m-1}{2mi}\right)^{2m} \sum_{p=0}^{\ell} g_{\ell,p,1}(\rho)v_{p}(t),
	\end{align*}
	where
	\begin{align}\label{A_P-1_0}
	g_{\ell,0,1}(\rho) = ( \text{\textcursive{ p}}_{1,0} + \text{\textcursive{ p}}_{1,1} \delta_{0}^{0,1}) g_{\ell-1, 0}^{(2m-1)}(\rho)
	+ \text{\textcursive{ p}}_{1,1} \delta_{0}^{1,1} g_{\ell-1,1}^{(2m-1)}(\rho), \qquad\quad
	\end{align}
	\begin{align}\label{A_P-1-p}
	g_{\ell,p,1}(\rho) =  
	&\text{\textcursive{p}}_{1,1}\delta_{p}^{p-1,1} g_{\ell-1,p-1}^{(2m-1)}(\rho)
	+ g_{\ell-1,p}^{(2m-1)}(\rho)( \text{\textcursive{p}}_{1,0}\!\! +\! \text{\textcursive{p}}_{1,1}  \delta_{p}^{p,1}) g_{\ell-1,p}^{(2m-1)}(\rho)
	\\ 
	\nonumber
	& + \text{\textcursive{p}}_{1,1}\delta_{p}^{p+1,1} g_{\ell-1,p-1}^{(2m-1)}(\rho), \qquad p=1,2,\dots \ell-2,
	\end{align}
	\begin{align}\label{A-P-1-last-1}
	 g_{\ell,\ell-1,1}(\rho) =  \text{\textcursive{ p}}_{1,1}\delta_{\ell-1}^{\ell-1,1} g_{\ell-1,\ell-2}^{(2m-1)}(\rho)+
	( \text{\textcursive{ p}}_{1,0} + \text{\textcursive{ p}}_{1,1}  \delta_{\ell-1}^{\ell-1,1}) g_{\ell-1,\ell-1}^{(2m-1)}(\rho),
	\end{align}
	\begin{align}\label{A-P-1-last}
	g_{\ell,\ell,1}(\rho) = \text{\textcursive{ p}}_{1,1}\delta_{\ell}^{\ell-1,1} g_{\ell-1,\ell-1}^{(2m-1)}(\rho);\hspace{17em}
	\end{align}
    the symbols $\delta_{\nu}^{j,p}$ are defined in (\ref{Rec_Lag_2}),
    Lemma \ref{Rk-Ric-Rel}, in the Appendix. In the case $\ell = 1$ we have
    \begin{align*}
      \mathcal{P}_{1} u_{0}
	& = t^{2n}\! \left(\frac{2m-1}{2mi}\right)^{2m}\!
	\left[\left(\text{\textcursive{p}}_{1,0}\!\!  + \!\! \text{\textcursive{ p}}_{1,1} \delta_{0}^{0,1}\right) g_{0,0}^{(2m-1)}(\rho)v_{0}(t)
\right.
      \\
        &\phantom{=}
\left.   +  \text{\textcursive{ p}}_{1,1}\! \delta_{1}^{0,1}  g_{0,0}^{(2m-1)}(\rho)v_{1}(t)\right].
	\end{align*}
	Since the term $\text{\textcursive{ p}}_{1,0} $ is linear with
        respect to $r$ with a non zero coefficient,
        we may make a suitable choice of the parameter $r$ cancelling
        the coefficient of $ v_{0} $ in the above expression.
	\item[iii)] case of $\mathcal{P}_{i}$, $i= 2,3,\dots, 2m$; we have
          \begin{multline*}
        \mathcal{P}_{i} u_{\ell-i}  
        =
        t^{2n} \left(\frac{2m-1}{2mi}\right)^{2m}\mathscr{P}_{i}(t\partial_{t})\partial_{\rho}^{2m-i} \left(\sum_{p=0}^{\ell-i} g_{\ell-i,p}(\rho) v_{p}(t)\right)
\\
        = t^{2n} \left(\frac{2m-1}{2mi}\right)^{2m} \sum_{p=0}^{\ell} g_{\ell,p,i}(\rho) v_{p},
	\end{multline*}
	where
	\begin{align}\label{A-P_i}
	g_{\ell,p,i}(\rho) = \sum_{\nu = \max\lbrace p-i, 0\rbrace}^{\min\lbrace p+i, \ell -i\rbrace } g_{\ell-i,\nu}^{(2m-i)}(\rho)
	\left(\sum_{j=|p-\nu|}^{i}\text{\textcursive{ p}}_{i,j} \delta_{p-\nu}^{\nu,j} \right).
	\end{align}
    The symbols $\delta_{\nu}^{j,p}$ are defined in (\ref{Rec_Lag_2}),
    Lemma \ref{Rk-Ric-Rel} in the Appendix, and we have set
    $\delta_{p}^{p,0} = 1$. 
\end{enumerate}
We introduce the operator $\Pi_{0}$ and its ``orthogonal''
$(1-\Pi_{0})$ acting on the functions $u_{\ell}$ in the following way
$$
\Pi_{0}u_{\ell}= g_{0,\ell}(\rho) v_{0}(t) \quad \text{ and } \quad (1-\Pi_{0})u_{\ell} =
\sum_{p=1}^{\ell} g_{\ell,p}(\rho) v_{p}.
$$
As consequence of the choice of the parameter $r$ we have that
$$
\Pi_{0}\mathcal{P}_{1}\Pi_{0}u_{0}= \Pi_{0}\mathcal{P}_{1} u_{0} =  0 .
$$
Moreover
$$
\Pi_{0}\mathcal{P}_{1}\Pi_{0}u_{\ell}=0 \quad \text{ for every } \ell.
$$
This is crucial in order to obtain the right growth
of the functions $g_{\ell,0}(\rho)$ with respect to (negative) powers
of $\rho$ (see Lemma \ref{L-0_Est} in the Appendix).

Following the idea in \cite{M81}, to obtain the $u_{\ell}$, we
consider the system 
\begin{align}
  \label{transp}
\begin{cases}
\left(1-\Pi_{0}\right)\mathcal{P}_{0}u_{\ell} = -\left(1-\Pi_{0}\right)\displaystyle\sum_{i=1}^{\min\lbrace\ell, 2m\rbrace}\frac{1}{\rho^{i}}\mathcal{P}_{i}u_{\ell-i}; \\
\Pi_{0}\mathcal{P}_{0}u_{\ell} = -\displaystyle\frac{1}{\rho}\Pi_{0}\mathcal{P}_{1} \left(1-\Pi_{0}\right)u_{\ell} - \Pi_{0}\displaystyle\sum_{i=2}^{\min\{\ell+1,2m\}}\frac{1}{\rho^{i}}\mathcal{P}_{i}u_{\ell+1-i} .
\end{cases}
\end{align}
This allows us, at any step of the process, to reduce the problem of
computing the $ u_{\ell} $ to that of solving a system of $\ell+1$
ordinary differential equations yielding the functions $g_{\ell,p}(\rho)$,
$p = 0,\dots \ell$.

With the purpose of understanding how to construct and subsequently estimate the functions $g_{\ell,p}(\rho)$, we begin to analyze the case $\ell=1$.

For $ \ell = 1 $ system \eqref{transp} in terms of the $ g_{1, p} $, $
p = 0, 1 $, becomes
\begin{align*}
\begin{cases}
\Theta_{1} g_{1,1}(\rho) = -\displaystyle\frac{1}{\rho}
\text{\textcursive{p}}_{1,1}\delta_{1}^{0,1} g_{0,0}^{(2m-1)}(\rho);
\\[16pt]
\Theta_{0}g_{1,0} (\rho)= -\displaystyle\frac{1}{\rho}\text{\textcursive{p}}_{1,1}\delta^{1,1}_{0}g_{1,1}^{(2m-1)}(\rho)\!
- \displaystyle\frac{1}{\rho^{2}} 
\left(\text{\textcursive{p}}_{2,0} \delta_{0}^{0,0}+ \text{\textcursive{p}}_{2,1}\delta^{0,1}_{0}+\text{\textcursive{p}}_{2,2}\delta^{0,2}_{0}\right) 
g_{0,0}^{(2m-2)}(\rho),
\end{cases}
\end{align*}
where $ g_{0,0}(\rho) = e^{- c_{1} \rho} $.

We stress the fact that the choice of the parameter $r$ has allowed us
to make the right hand side of the second equation of order $-2$ with
respect to the variable  $\rho$.

We denote by $f_{1,j}(\rho)$, $j=1,0$, the functions on the right hand
side of the above system.

In order to be able to apply Lemmas \ref{L-j_Est}, \ref{L-0_Est}, we
need to make sure that the variable $ \sigma $ in those Lemmas belongs
to the half line $ [C_{0}(j+1), +\infty[ $. To accomplish this, we
use cutoff functions $ \chi_{\ell} $ so that the hypotheses of those
Lemmas are satisfied.

We define a family of smooth functions $\lbrace
\chi_{\ell}(\rho)\rbrace_{\ell \geq 0}$ such that


%
\begin{enumerate}
	\item[i)] $\chi_{\ell}(\rho)$ is identically zero for $\rho <
          2 R_{1} (\ell+1)$ and identically one for $\rho >
          4 R_{1}(\ell +1)$, where $ R_{1}$ denotes a suitable
          positive constant; 
	\item[ii)]there is a constant $C_{\chi}$, independent of $\rho $
          and $\ell $, such that 
          \begin{align}
            \label{ii}
	|\chi_{\ell}^{(k)}(\rho)| \leq C_{\chi} \qquad \forall k\leq 2m.
	\end{align}
\end{enumerate}
%
%
%
We define the functions $g_{1,p}(\rho)$, $p=0,1$, as
\begin{align*}
g_{1,1}(\rho)= \left(G_{1}\ast (\chi_{1} f_{1})\right)(\rho),
\end{align*}
and
\begin{align*}
g_{1,0}(\rho)=  \left( G_{0}\ast (\chi_{0} f_{0}) \right) (\rho)- h_{0}(\rho)-h_{m-1}(\rho),
\end{align*}
where $G_{0}$ and $G_{1}$ are the fundamental solutions of $\Theta_{0}$ and $\Theta_{1}$, see Lemma \ref{L_Fund_Sol}, and
the functions  $h_{0}(\rho)$ and $h_{m-1}(\rho)$ are the solutions of
the linear homogeneous equation $\Theta_{0} h=0$ defined in Lemma
\ref{L-0_Est}.

By Lemma \ref{L-j_Est}, see also Remark \ref{Rk_L-j_Est}, and Lemma
\ref{L-0_Est}, see also (\ref{Der-Fun-odd}), \eqref{Der-Fun-even} in
Lemma \ref{L_Fund_Sol}, we obtain 
\begin{align*}
|g_{1,p}(\rho)| \leq C^{2}  \frac{1}{\rho} e^{-c_{0}\rho}, \qquad
  p= 0, 1,
\end{align*}
for $\rho \geq R_{1} $, where $C$ is a positive constant depending on
$m$ and $n$.

\medskip

Let us now consider the case $ \ell > 1 $. To construct the functions $g_{\ell,p}(\rho)$ and consequently the functions $u_{\ell}(t,\rho)$,
we proceed recursively adopting the technique described above. 


We have to solve the following system of $\ell +1$ equations
\begin{align}
  \begin{cases}
    \label{system-ell}
\Theta_{0}g_{\ell,0}(\rho) = f_{\ell,0} (\rho);  \\
\Theta_{1} g_{\ell,1}(\rho) =f_{\ell,1}(\rho);\\
\vdots\\
\Theta_{\ell}g_{\ell,\ell}(\rho) = f_{\ell,\ell}(\rho).
\end{cases}
\end{align}
Where due to (\ref{A_P-1_0}), (\ref{A_P-1-p}), (\ref{A-P-1-last-1}), (\ref{A-P-1-last}) and (\ref{A-P_i}) we have
\begin{align}
  \label{flp}
f_{\ell,p}(\rho)= -\sum_{i=1}^{2m} \frac{1}{\rho^{i}}g_{\ell,p,i}(\rho), \quad p=1,\dots, \, \ell,
\end{align}
or, more explicitly, for $ p=1,\dots, \, \ell, $
$$ 
f_{\ell,p}(\rho) = \sum_{i=1}^{\min\{\ell, 2m\}} \frac{1}{\rho^{i}}
\left(\mathcal{P}_{i} u_{\ell - i} \right)_{p} =
\sum_{i=1}^{\min\{\ell, 2m\}} \frac{1}{\rho^{i}} 
\sum_{j=0}^{i} \sum_{p_{1}=0}^{\ell-i} \text{\textcursive{ p}}_{i,j}
\delta^{p, j}_{p-p_{1}} g^{(2m-i)}_{\ell-i, p_{1}} ,
$$
where $ \left(\mathcal{P}_{i} u_{\ell - i} \right)_{p} $ denotes the
coefficient of $ v_{p} $ in the expression of $ \mathcal{P}_{i}
u_{\ell - i}  $.

Moreover
\begin{align}
  \label{fl0}
f_{\ell,0}(\rho)=- \frac{1}{\rho}g_{\ell,1}^{(2m-1)}(\rho)\text{\textcursive{p}}_{1,1} \delta_{0}^{1,1}
-\sum_{i=2}^{\min\{\ell, 2m\}}\sum_{\nu = 0}^{\min\lbrace i, \ell -i+1\rbrace } \frac{1}{\rho^{i}} g_{\ell-i+1,\nu}^{(2m-i)}(\rho)
\left(\sum_{j=\nu}^{i}\text{\textcursive{ p}}_{i,j} \delta_{-\nu}^{\nu,j} \right).
\end{align}
To comply with the hypotheses of Lemmas \ref{L-j_Est}, \ref{L-0_Est},
we solve instead the system
\begin{align}
  \begin{cases}
    \label{system-ell-cutoff}
\Theta_{0}g_{\ell,0}(\rho) = f_{\ell,0} (\rho) \chi_{\ell} ;  \\
\Theta_{1} g_{\ell,1}(\rho) =f_{\ell,1}(\rho) \chi_{\ell}  ;  \\
\vdots\\
\Theta_{\ell}g_{\ell,\ell}(\rho) = f_{\ell,\ell}(\rho) \chi_{\ell} .
\end{cases}
\end{align}
We remark that all the sums involved in the above formulas, see
(\ref{A-P_i}), are finite and they have at most $4m $ terms.
We also point out that, due to the recurrence relation described in
Lemma \ref{Rk-Ric-Rel} in the Appendix, 
we have that $|\delta_{p}^{\nu,j}| \leq
C^{j}\frac{(\nu+j)!}{\nu!}$. The function $f_{\ell,0}(\rho)$ has order
$-\ell-1$ with respect to $\rho$.

The system \eqref{system-ell-cutoff} is then solved by
\begin{equation}
\label{syst-sol}
\begin{cases}
g_{\ell, p}(\rho) = G_{p} * (\chi_{\ell} f_{\ell, p})(\rho), \quad p =
1,\ldots, \ell \\[5pt]
g_{\ell, 0}(\rho) = G_{0} * (\chi_{\ell} f_{\ell, 0})(\rho) -
h_{0}(\rho) - h_{m-1}(\rho) ,
\end{cases}
\end{equation}
where the functions  $h_{0}(\rho)$ and $h_{m-1}(\rho)$ are the solutions ofthe linear homogeneous equation $\Theta_{0} h=0$ defined in Lemma
\ref{L-0_Est}

The next lemma gives the estimates of the derivatives of the
coefficients $ g_{\ell, p} $.
\begin{lemma}
\label{lemma:glp}
With the above notations we have that
\begin{align}\label{Est_g-lp-1}
|g_{\ell,p} (\rho)| \leq C^{\ell+1} \frac{(\ell+ 1)^{\ell
  \left(1-\frac{1}{2m}\right)}}{\rho^{\ell}} e^{-c_{0}\rho},\qquad
  0\leq p \leq \ell ,
\end{align} 
and more generally,
\begin{align}\label{Est-g-lp-2}
|g_{\ell,p}^{(k)} (\rho)| \leq C^{\ell+1+(\frac{k}{2m} - 1)_{+}} \frac{(\ell+ 1)^{\ell
  \left(1-\frac{1}{2m}\right) + \frac{k}{2m}}}{\rho^{\ell}}
  e^{-c_{0}\rho},\qquad 0\leq p \leq \ell ,
\end{align} 
with $ k \in \mathbb{Z}_{+} $. Here $ x_{+} = x $ if $ x \geq 0 $ and
$ x_{+} = 0 $ if $ x < 0 $. 
\end{lemma}
\begin{proof}
We apply Lemmas \ref{L-j_Est}, \ref{L-0_Est} as well as Remark
\ref{Rk_L-j_Est}. We have
\begin{align}
  \label{glp}
	g_{\ell,p}^{(k)}(\rho)=G_{p}^{(k)}\ast \left(-   \chi_{\ell}(\sigma)
  \sum_{i=1}^{\min\lbrace\ell, 2m\rbrace} \frac{1}{\sigma^{i}}
  g_{\ell,p,i}(\sigma) \right)(\rho), 
	\quad \text{ for } p\geq 1,
	\end{align}
	where 
	\begin{align}
          \label{glpi}
	g_{\ell,p,i}(\sigma)= \sum_{\nu = \max\lbrace p-i,
          0\rbrace}^{\min\lbrace p+i, \ell -i\rbrace }  g_{\ell-i,\nu}^{(2m-i)}(\sigma)
	\left(\sum_{j=|p-\nu|}^{i}\text{\textcursive{ p}}_{i,j}
          \delta_{p-\nu}^{\nu,j} \right) .
	\end{align}
We proceed by induction on $ \ell $.
Assume that for every $ \ell' < \ell $, for every $ k \in
\mathbb{Z}_{+} $ and for every $ 0 \leq p \leq \ell'  $, we have
$$
|g^{(k)}_{\ell',p}(\rho)| \leq C^{\ell'+1+(\frac{k}{2m} - 1)_{+}}
\left(\ell' +1\right)^{\ell'(1-1/2m) + \frac{k}{2m}} \rho^{-\ell'} e^{-c_{0}\rho} .
$$
Let us estimate first the quantity in \eqref{glpi}. We have
\begin{multline} 
| g_{\ell,p,i}(\sigma) | \leq \sum_{\nu = \max\lbrace p-i,
          0\rbrace}^{\min\lbrace p+i, \ell -i\rbrace }  C^{\ell - i+1}
\left(\ell-i +1\right)^{(\ell-i)(1-1/2m) + \frac{2m-i}{2m}}
\sigma^{-(\ell-i)} e^{-c_{0}\sigma}
\\
\cdot
C_{\text{\textcursive{p}}}  \sum_{j=|p-\nu|}^{i} C_{1}^{j}
\nu^{j} .
\end{multline}
Hence
\begin{multline*}
| g_{\ell, p}^{(k)}(\rho) | \leq C_{\text{\textcursive{p}}} \sum_{i=1}^{\min\lbrace\ell,
  2m\rbrace} \sum_{\nu = \max\lbrace p-i,
  0\rbrace}^{\min\lbrace p+i, \ell -i\rbrace }
\sum_{j=|p-\nu|}^{i}
C_{1}^{j} \nu^{j}
\\
C^{\ell - i+1}
\left(\ell-i +1\right)^{(\ell-i)(1-1/2m) + \frac{2m-i}{2m}}
\int_{R}^{+\infty} | G_{p}^{(k)}(\rho -
          \sigma) | \frac{1}{\sigma^{\ell}} e^{- c_{0} \sigma}
          d\sigma
\\
\leq
C_{\text{\textcursive{p}}} C_{G} \sum_{i=1}^{\min\lbrace\ell,
  2m\rbrace} \sum_{\nu = \max\lbrace p-i,
  0\rbrace}^{\min\lbrace p+i, \ell -i\rbrace }
\sum_{j=|p-\nu|}^{i} C_{1}^{j} \nu^{j}
C^{\ell - i+1}
\left(\ell-i +1\right)^{(\ell-i)(1-1/2m) + \frac{2m-i}{2m}}
\\
\cdot
\left(\frac{2m-1}{2m}\right)^{2m-1-k} E_{p}^{\frac{k+1}{2m} - 1}
\int_{R}^{+\infty} e^{-c_{p} | \rho - \sigma|}
  \frac{1}{\sigma^{\ell}} e^{- c_{0} \sigma}
          d\sigma .
\end{multline*}
Let us now use Lemma \ref{L-j_Est} as well as the fact that the
indices $ j $, $ \nu $, $ i $ run over a finite number of values.

We point out that there are two positive constants $ 1 < C_{-} < C_{+} $
such that
$$
C_{-} (p+1) \leq E_{p} \leq C_{+} (p+1) .
$$
As a consequence $ E_{p}^{\theta} \leq C_{+}^{\theta} (p+1)^{\theta} $ if $
\theta \geq 0 $, whereas $ E_{p}^{\theta} \leq C_{-}^{\theta}
(p+1)^{\theta} \leq (p+1)^{\theta} $ if $ \theta < 0 $. Setting $ C_{E} =
C_{+}  $ we have that $ E_{p}^{\theta} \leq
C_{E}^{\theta_{+}} (p+1)^{\theta} $, where $ \theta_{+} = \theta $ for
$ \theta \geq 0 $, $ \theta_{+} = 0 $ for $ \theta < 0 $. We
note that $ C_{E} $ depends only on the problem data.

We then obtain
\begin{multline*}
| g_{\ell, p}^{(k)}(\rho) | \leq
C_{2} \frac{e^{- c_{0} \rho}}{\rho^{\ell}}  \sum_{i=1}^{\min\lbrace\ell,
  2m\rbrace} \sum_{\nu = \max\lbrace p-i,
  0\rbrace}^{\min\lbrace p+i, \ell -i\rbrace }
\left(\sum_{j=0}^{i} C_{1}^{j} \right) C^{\ell - i+1}
\\
\cdot
(\ell - i + 1)^{(\ell - i) (1 - \frac{1}{2m}) + \frac{2m - i}{2m} + i
}
\ C_{E}^{(\frac{k}{2m} -1)_{+}} (p+1)^{\frac{k}{2m} - 1}
\\
\leq
C^{\ell+1} (\ell+1)^{\ell(1 - \frac{1}{2m}) + \frac{k}{2m}} \frac{e^{-
    c_{0} \rho}}{\rho^{\ell}}
C_{2} C_{3} C_{E}^{(\frac{k}{2m} -1)_{+}} \sum_{i=1}^{\min\lbrace\ell,
  2m\rbrace} C^{-i}
\\
\leq
C^{\ell+1+(\frac{k}{2m} - 1)_{+}} (\ell+1)^{\ell(1 - \frac{1}{2m}) + \frac{k}{2m}} \frac{e^{-
    c_{0} \rho}}{\rho^{\ell}} ,
\end{multline*}
provided $ C $ is chosen suitably large.

\end{proof}

So far we proved the following
\begin{proposition}
\label{formal}
There are functions $ u_{\ell} $, $ \ell = 0, 1, \ldots $, such that
\begin{equation}
  \label{uell}
u_{\ell}(t, \rho) = \sum_{p=0}^{\ell} g_{\ell, p}(\rho) v_{p}(t) ,
\end{equation}
where the $ v_{p} $ are the eigenfunctions of the operator in
\eqref{Eig_P} defined in \eqref{Ev_Eigf}. $ g_{0, 0} $ is defined in
\eqref{u0} and the $ g_{\ell, p}(\rho) $ are defined by
\eqref{syst-sol}, \eqref{flp}, \eqref{fl0} and satisfy the estimates
\eqref{Est-g-lp-2}. Finally  the
function
$$ 
u(t, \rho) = \sum_{\ell \geq 0} u_{\ell}(t, \rho) 
$$
is a formal solution of \eqref{Pu=0} in the sense that system
\eqref{transp} is verified and formally equivalent to
\eqref{Pu=0}.

We say that $ \mathscr{K}[u] $ is a formal solution of equation
\eqref{Op-GM}. 
\end{proposition}
Next we need a lemma allowing us to estimate the derivatives with
respect to $ t $ of the functions $ u_{\ell} $, $ \ell \geq 0 $.
\begin{lemma}
\label{Dul}
For any $ \ell \in \mathbb{N} $, any $ \alpha $, $ \beta $, $ \gamma
$, we have the estimate
\begin{equation}
\label{Dul-est}
|t^{\alpha} \partial_{t}^{\beta} \partial_{\rho}^{\gamma} u_{\ell}(t,
\rho) | \leq C_{u}^{\ell+\alpha+\beta+\gamma+1} (\ell+1)^{\ell
  \frac{2m-1}{2m} + \frac{\gamma}{2m}} \alpha!^{\frac{1}{2n+2}}
\beta!^{\frac{2n+1}{2n+2}} \frac{1}{\rho^{\ell}} e^{- c_{0} \rho},
\end{equation}
where $ C_{u} $ denotes a positive constant independent of $ \ell $, $
\alpha $, $ \beta $, $ \gamma $. 
\end{lemma}
\begin{proof}
We apply Lemma \ref{lemma:glp} and Theorem \ref{G-S_Char} to
$$ 
u_{\ell}(t, \rho) = \sum_{p=0}^{\ell} g_{\ell,p}(\rho) v_{p}(t) .
$$
Then
\begin{multline*}
|t^{\alpha} \partial_{t}^{\beta} \partial_{\rho}^{\gamma} u_{\ell}(t,
\rho) | \leq \sum_{p=0}^{\ell} | \partial_{\rho}^{\gamma}
g_{\ell,p}(\rho)| | t^{\alpha} \partial^{\beta}_{t} v_{p}(t) |
\\
\leq
\sum_{p=0}^{\ell} C_{v}^{p+\alpha+\beta+1}
\alpha!^{\frac{1}{2n+2}} \beta!^{\frac{2n+1}{2n+2}}
C_{g}^{\ell+1+(\frac{k}{2m} - 1)_{+}} \frac{(\ell+ 1)^{\ell
  \left(1-\frac{1}{2m}\right) + \frac{k}{2m}}}{\rho^{\ell}}
  e^{-c_{0}\rho}
\end{multline*}
Then we reach the conclusion choosing $ C_{u} $ large enough.
\end{proof}

\section{Turning a formal solution into a true solution} 
\renewcommand{\theequation}{\thesection.\arabic{equation}}
\setcounter{equation}{0} \setcounter{theorem}{0}

Our next task is to turn the formal solution just constructed into a
true solution. The obtained function is a solution
of an equation of the form $ M_{n, m} u = f $, with $ f $, more
regular than $ G^{\frac{2m}{2m-1}} $.

In order to define the ``approximate'' solution to the problem
$M_{n,m}u=0$, we need another family of cutoff functions.
\begin{lemma}[\cite{bmt}]
\label{lemma:cutoff}
Let $ \sigma > 1 $. There exists a family of cutoff functions $ \omega_{j}
\in G^{\sigma}(\mathbb{R}^{n}_{x}) $, $ 0 \leq
\omega_{j} (x) \leq 1 $, $ j = 0, 1, 2, \ldots $, such that
\begin{itemize}
\item[1- ]{}
$ \omega_{j} \equiv 0 $ if $ |x| \leq 2 R (j+1) $, $ \omega_{j} \equiv
1$ if $ |x| \geq 4 R (j+1) $, with $ R $ an arbitrary positive
constant. 
\item[2- ]{}
There is a suitable constant $ C_{\omega} $, independent of $ j $, $
\alpha $, $ R $, such that
\begin{equation}
\label{eq:cutoffder2}
| D^{\alpha} \omega_{j}(x) | \leq C_{\omega}^{|\alpha| + 1}
R^{-|\alpha|}, 
\qquad \text{if } |\alpha| \leq 3j.
\end{equation}
and
\begin{equation}
\label{eq:cutoffdernolim2}
| D^{\alpha} \omega_{j}(x) | \leq ( R C_{\omega})^{|\alpha| + 1}
\frac{\alpha!^{\sigma}}{|x|^{|\alpha|}}, \qquad \text{for every } \alpha.
\end{equation}
\end{itemize}
\end{lemma}
Let $ \omega_{\ell} $ denote cutoff functions like those in the above
lemma, with $ \sigma > 1 $ sufficiently close to 1 and $ R $ large to
be chosen later. We shall be more precise further on.

Define 
\begin{align}
  \label{utilde}
\tilde{u}(t,\rho) =\sum_{\ell\geq 0} u_{\ell}(t,\rho)
  \omega_{\ell}(\rho),
\end{align}
where $ u_{\ell} $ is given by Proposition \ref{formal}. We point out
that the above series is actually a locally finite sum, so that we do
not have any convergence problem. 

We consider
\begin{equation}
\label{Kutilde}
\mathscr{K}[\tilde{u}](x,y) 
=\int_{0}^{+\infty}  e^{iy\rho^{\frac{2m}{2m-1}}}
\rho^{r+\frac{2m}{(n+1)(2m-1)}}  \tilde{u} (t,\rho)_{\big|{
    t=\rho^{\frac{m}{(n+1)(2m-1)}}x}} d\rho ,
\end{equation}
where $ \tilde{u} $ is given by \eqref{utilde}. 

Applying $M_{n,m}$ to $ \mathscr{K}[\tilde{u}](x,y) $, we obtain
$$ 
\mathscr{K}\left[ \sum_{i=0}^{2m} \rho^{-i} \mathcal{P}_{i} (t,
  \partial_{t}, \partial_{\rho}) \sum_{\ell \geq 0} u_{\ell}
  \omega_{\ell} \right](x,y) .
$$
Let us look at the quantity in square brackets. We have
$$
\sum_{i=0}^{2m} \rho^{-i} \mathcal{P}_{i} (t,
  \partial_{t}, \partial_{\rho}) \sum_{\ell \geq 0} u_{\ell}
  \omega_{\ell}
  =
\sum_{i=0}^{2m} \sum_{\ell \geq 0} \sum_{\gamma=0}^{2m}
\frac{1}{\rho^{i}} \partial_{\rho}^{\gamma} \omega_{\ell}
\frac{1}{\gamma!} \mathcal{P}_{i}^{(\gamma)}(t, \partial_{t},
\partial_{\rho}) u_{\ell},   
$$
where $ \mathcal{P}_{i}^{(\gamma)}  $denotes the differential operator
whose symbol is $ \partial_{\sigma}^{\gamma}\mathcal{P}_{i}(t, \tau,
\sigma) $.

We are going to consider separately the cases where the cutoff
function $ \omega_{\ell} $ takes derivatives and those where it does not
take any derivative. The above quantity becomes
$$ 
\sum_{i=0}^{2m} \sum_{\ell \geq 0}
\frac{1}{\rho^{i}} \omega_{\ell} \mathcal{P}_{i}(t, \partial_{t},
\partial_{\rho}) u_{\ell}
+
\sum_{i=0}^{2m} \sum_{\ell \geq 0} \sum_{\gamma=1}^{2m}
\frac{1}{\rho^{i}} \partial_{\rho}^{\gamma}\omega_{\ell}
\frac{1}{\gamma!} \mathcal{P}_{i}^{(\gamma)}(t, \partial_{t},
\partial_{\rho}) u_{\ell} = A_{1} + A_{2}.
$$
Let us consider $ A_{1} $ first. We have
\begin{multline}
  \label{A1}
A_{1} = \sum_{\ell \geq 0} \omega_{\ell}  \mathcal{P}_{0} u_{\ell} + \sum_{\ell \geq 0} \sum_{i=1}^{2m} 
\frac{1}{\rho^{i}} \omega_{\ell} \mathcal{P}_{i}(t, \partial_{t},
\partial_{\rho}) u_{\ell}
\\
=
\sum_{\ell \geq 0} \omega_{\ell}  \mathcal{P}_{0} u_{\ell} +
\sum_{\ell \geq 1} \sum_{i=1}^{\min\{\ell, 2m\}}  \frac{1}{\rho^{i}}
\omega_{\ell-i}  \mathcal{P}_{i}(t, \partial_{t},
\partial_{\rho}) u_{\ell - i} 
\\
=
\sum_{\ell \geq 1} \omega_{\ell} ( 1 - \Pi_{0}) \mathcal{P}_{0}
u_{\ell} +
\sum_{\ell \geq 1} \sum_{i=1}^{\min\{\ell, 2m\}}  \frac{1}{\rho^{i}}
\omega_{\ell-i} (1 - \Pi_{0}) \mathcal{P}_{i}(t, \partial_{t},
\partial_{\rho}) u_{\ell - i}
\\
+
\sum_{\ell \geq 1} \omega_{\ell}  \Pi_{0} \mathcal{P}_{0}
u_{\ell} +
\sum_{\ell \geq 1} \sum_{i=1}^{\min\{\ell, 2m\}}  \frac{1}{\rho^{i}}
\omega_{\ell-i} \Pi_{0}  \mathcal{P}_{i}(t, \partial_{t},
\partial_{\rho}) u_{\ell - i}
\\
=
\sum_{\ell \geq 1} \left[ \omega_{\ell} ( 1 - \Pi_{0}) \mathcal{P}_{0}
u_{\ell} +
\sum_{i=1}^{\min\{\ell, 2m\}}  \frac{1}{\rho^{i}}
\omega_{\ell-i} (1 - \Pi_{0}) \mathcal{P}_{i}(t, \partial_{t},
\partial_{\rho}) u_{\ell - i} \right]
\\
+
\sum_{\ell \geq 1} \omega_{\ell}  \Pi_{0} \mathcal{P}_{0}
u_{\ell} + \sum_{\ell \geq 2} \frac{1}{\rho} \omega_{\ell-1} \Pi_{0}
\mathcal{P}_{1} (1 - \Pi_{0}) u_{\ell-1}
+
\sum_{\ell \geq 2} \sum_{i=2}^{\min\{\ell, 2m\}} \frac{1}{\rho^{i}}
\omega_{\ell-i} \Pi_{0} \mathcal{P}_{i}(t, \partial_{t},
\partial_{\rho}) u_{\ell - i}
\\
=
\sum_{\ell \geq 1} \left[ \omega_{\ell} ( 1 - \Pi_{0}) \mathcal{P}_{0}
u_{\ell} +
\sum_{i=1}^{\min\{\ell, 2m\}}  \frac{1}{\rho^{i}}
\omega_{\ell-i} (1 - \Pi_{0}) \mathcal{P}_{i}(t, \partial_{t},
\partial_{\rho}) u_{\ell - i} \right]
\\
+
\sum_{\ell \geq 1} \left[ \omega_{\ell}  \Pi_{0} \mathcal{P}_{0}
  u_{\ell} +
\frac{1}{\rho} \omega_{\ell} \Pi_{0} \mathcal{P}_{1} (1 - \Pi_{0})
u_{\ell} +
\sum_{i=2}^{\min\{\ell+1, 2m\}} \frac{1}{\rho^{i}}
\omega_{\ell+1-i} \Pi_{0} \mathcal{P}_{i}(t, \partial_{t},
\partial_{\rho}) u_{\ell +1 - i}
\right] ,
\end{multline}
where we used the fact that $ \mathcal{P}_{0} u_{0} = 0 $.

We immediately see that if the functions $ \omega $ are identically 1
or 0 the quantities in brackets above are zero because of
\eqref{transp}. As a consequence the quantities in square brackets
above have $ \rho $--support for $ 2R(\ell+1-2m) \leq \rho \leq
4R(\ell + 1) $. 

Consider now $ A_{2} $. Since $ \gamma $ runs over a finite interval,
the derivatives of the functions $ \omega_{\ell} $ are uniformly
bounded and as a consequence
$$ 
\sum_{\gamma=1}^{2m} \frac{1}{\rho^{i}} \partial_{\rho}^{\gamma}\omega_{\ell}
\frac{1}{\gamma!} \mathcal{P}_{i}^{(\gamma)}(t, \partial_{t},
\partial_{\rho}) u_{\ell} 
$$
has $ \rho $--support in the interval $ [ 2R(\ell+1), 4R(\ell+1)]
$. Hence
$$
M_{n,m} \mathscr{K}[\tilde{u}](x,y) = \mathscr{K} \Big[\sum_{\ell \geq 1}
w_{\ell}\Big](x,y), 
$$
where the $ \rho $--support of $ w_{\ell} $ is contained in the
interval $ [ 2R(\ell+1-2m) , 4R(\ell+1)] $.

We now show that $ \mathscr{K}[w](x,y)$, where $ w = \sum_{\ell \geq
  1} w_{\ell} $ belongs to the following class of functions.
\begin{definition}
	Let $s>1$, we denote by $\gamma^{s}(\Omega)$, $\Omega$ open subset of $\mathbb{R}^{n}$, the set of all $\varphi \in C^{\infty}(\Omega)$,
	for which, to every compact set $K$ in $\Omega$ and every
        $\varepsilon > 0$, there is a constant $C_{\varepsilon}$ such that, for every
	$\alpha \in \mathbb{Z}_{+}^{n}$,
	\begin{align}
	|D^{\alpha}\varphi(x)| \leq C \varepsilon^{|\alpha|} |\alpha|^{s|\alpha|}, \quad x\in K.
	\end{align}
\end{definition}
The classes $\gamma^{s}(\Omega)$ are the (local) Beurling classes of order $ s
$ and, for further reference on the subject, we refer to
\cite{H63}. We may also define the global version of the classes $
\gamma $, by using uniform constants. These classes are denoted as $
\gamma_{g}^{s}(\Omega) $.

\begin{proposition}
\label{gamma}
With the above notation the function
$M_{n,m}\mathscr{K}[\tilde{u}](x,y)$  belongs to
$\gamma^{2m/(2m-1)}_{g}(\mathbb{R}^{2})$. 
\end{proposition}
\begin{proof}
We start off by estimating the derivatives of $ \mathscr{K} \Big[\sum_{\ell \geq 1}
w_{\ell}\Big](x,y) $. Set $ s_{0} = \frac{2m}{2m-1} $ and $ r' = r +
\frac{s_{0}}{n+1} $. Then
\begin{equation}
  \label{f}
| \partial_{y}^{\alpha} \partial_{x}^{\beta} \mathscr{K} \Big[\sum_{\ell \geq 1}
w_{\ell}\Big](x,y) |
= \left| \int_{0}^{+\infty}  e^{iy\rho^{s_{0}}}
\rho^{r' + \alpha s_{0} + \frac{\beta s_{0}}{2(n+1)}}
(\partial_{t}^{\beta} w(t,\rho))_{\big|{
    t=\rho^{\frac{s_{0}}{2(n+1)}}x}} d\rho \right|.
\end{equation}
As we pointed out above the quantity $ \partial_{t}^{\beta} w(t,\rho)
$ is the sum of two functions: $ \partial_{t}^{\beta} w(t,\rho) =
\partial_{t}^{\beta} f_{1}(t, \rho) + \partial_{t}^{\beta} f_{2}(t,
\rho) $, where $ f_{1} $ contains the terms where the cutoffs $
\omega_{\ell} $ are not getting derived, while $ f_{2} $ has the terms
where the $ \omega_{\ell} $ takes derivatives.
$$ 
f_{1}(t, \rho) = \sum_{i=0}^{2m} \sum_{\ell \geq 0}
\frac{1}{\rho^{i}} \omega_{\ell} \mathcal{P}_{i}(t, \partial_{t},
\partial_{\rho}) u_{\ell} ,
$$
$$ 
f_{2}(t, \rho) = \sum_{i=0}^{2m} \sum_{\ell \geq 0} \sum_{\gamma=1}^{2m}
\frac{1}{\rho^{i}} \partial_{\rho}^{\gamma}\omega_{\ell}
\frac{1}{\gamma!} \mathcal{P}_{i}^{(\gamma)}(t, \partial_{t},
\partial_{\rho}) u_{\ell} .
$$
Consider the above integral where $ \partial_{t}^{\beta} w $ has been
replaced by $ \partial_{t}^{\beta} f_{2} $. We have to estimate
\begin{multline*}
\left| \int_{0}^{+\infty}  e^{iy\rho^{s_{0}}}
\rho^{r' + \alpha s_{0} + \frac{\beta s_{0}}{2(n+1)}}
\left (\partial_{t}^{\beta} \sum_{i=0}^{2m} \sum_{\ell \geq 0} \sum_{\gamma=1}^{2m}
\frac{1}{\rho^{i}} \partial_{\rho}^{\gamma}\omega_{\ell}
\frac{1}{\gamma!} \mathcal{P}_{i}^{(\gamma)}(t, \partial_{t},
\partial_{\rho}) u_{\ell}  \right )_{\big|{
  t=\rho^{\frac{s_{0}}{2(n+1)}}x}} d\rho \right|
\\
\leq
\sum_{i=0}^{2m} \sum_{\ell \geq 0} \sum_{\gamma=1}^{2m}
\int_{0}^{+\infty} \rho^{r' + \alpha s_{0} + \frac{\beta
    s_{0}}{2(n+1)}}
\left |\partial_{t}^{\beta} \frac{1}{\rho^{i}} \partial_{\rho}^{\gamma}\omega_{\ell}
\frac{1}{\gamma!} \mathcal{P}_{i}^{(\gamma)}(t, \partial_{t},
\partial_{\rho}) u_{\ell}  \right |_{\big|{
  t=\rho^{\frac{s_{0}}{2(n+1)}}x}} d\rho
\\
\leq
\sum_{i=0}^{2m} \sum_{\ell \geq 0} \sum_{\gamma=1}^{2m} \sum_{\beta'
  \leq \beta} \binom{\beta}{\beta'}
\int_{0}^{+\infty} \rho^{r' + \alpha s_{0} + \frac{\beta
    s_{0}}{2(n+1)}}
\left | \frac{1}{\rho^{i}} \partial_{\rho}^{\gamma}\omega_{\ell}
\frac{1}{\gamma!} \mathcal{P}_{i, (\beta')}^{(\gamma)}(t, \partial_{t},
\partial_{\rho}) \partial_{t}^{\beta - \beta'} u_{\ell}(t, \rho)  \right |_{\big|{
  t=\rho^{\frac{s_{0}}{2(n+1)}}x}} d\rho ,
\end{multline*}
Where $  \mathcal{P}_{i, (\beta')}^{(\gamma)}(t, \partial_{t},
\partial_{\rho}) $ denotes the differential operator whose symbol is
given by $ \partial_{\sigma}^{\gamma} \partial_{t}^{\beta'} \mathcal{P}_{i}(t, \tau,
\sigma) $. 

Next we need to bound $ | \partial_{t}^{\beta} f_{2} | $. This
derivative is a sum of terms of the form
$$ 
\binom{\beta}{\beta'} \frac{1}{\rho^{i}}
|\partial_{\rho}^{\gamma}\omega_{\ell}| | t^{j-\beta'}
\partial_{t}^{j+\beta-\beta'} \partial_{\rho}^{2m-i-\gamma}
u_{\ell}(t, \rho)| ,
$$
modulo a constant independent of $ \alpha $, $ \beta $. We apply Lemma
\ref{Dul} to obtain the bound for the above quantity
\begin{multline*}
 \binom{\beta}{\beta'} \frac{1}{\rho^{i}}
|\partial_{\rho}^{\gamma}\omega_{\ell}|
C_{u}^{\ell+j-\beta'+j+\beta-\beta'+2m-i-\gamma+1}
\\
\cdot
(\ell+1)^{\ell
  \frac{2m-1}{2m} + \frac{2m-i-\gamma}{2m}}
(j-\beta')!^{\frac{1}{2n+2}} (j+\beta-\beta')!^{\frac{2n+1}{2n+2}}
\frac{e^{-c_{0}\rho}}{\rho^{\ell}} . 
\end{multline*}
Since $ 0 \leq i \leq 2m $, $ 1 \leq \gamma \leq 2m-i $, $ j \leq i $,
we may bound $ j $ by $ i $ and replace $ \beta' $ by zero. We also
point out that the base of each factorial in the second line above is
bounded by $ \ell + \beta + 2m $. As a consequence we get the estimate
\begin{multline*}
\binom{\beta}{\beta'}
\frac{e^{-c_{0} \rho}}{\rho^{i+\ell}} |\partial_{\rho}^{\gamma}\omega_{\ell}|
C_{1}^{\ell+\beta+1} (\ell + \beta + 2m)^{\ell \frac{2m-1}{2m} + i
  \frac{2m-1}{2m} + \frac{2m-\gamma}{2m} + \beta \frac{2n+1}{2n+2}}
\\
\leq
\frac{e^{-c_{0} \rho}}{\rho^{i+\ell}}
|\partial_{\rho}^{\gamma}\omega_{\ell}|
C_{2}^{\ell+\beta+1} (\ell + \beta + 2m)^{\ell \frac{2m-1}{2m} + \beta
  \frac{2n+1}{2n+2}}
\\
\leq
\frac{e^{-c_{0} \rho}}{\rho^{i+\ell}}
|\partial_{\rho}^{\gamma}\omega_{\ell}|
C_{3}^{\ell+\beta+1} \ell^{\ell \frac{2m-1}{2m}} \beta^{
  \beta\frac{2n+1}{2n+2}} .
\end{multline*}
Now the support of $ \partial_{\rho}^{\gamma}\omega_{\ell} $ is
contained in the interval $ [2R (\ell+1) , 4R (\ell+1)] $, so that $
\ell \leq \frac{\rho}{2R} - 1 $. As a consequence
$$ 
\frac{1}{\rho^{\ell}} \ell^{\ell \frac{2m-1}{2m}} \leq
\left(\frac{1}{2R}\right)^{\ell \frac{2m-1}{2m}} \rho^{-
  \frac{\ell}{2m}} \leq \left(\frac{1}{2R}\right)^{\ell
  \frac{2m-1}{2m}} \rho^{- \rho \frac{1}{4mR} + \frac{1}{2m}} .
$$
This allows us to conclude that
$$ 
| \partial_{t}^{\beta} f_{2}(t, \rho) | \leq \sum_{\ell \geq 0} \left(\frac{C_{4}}{2R}\right)^{\ell
  \frac{2m-1}{2m}} C_{4}^{\beta+1} \beta^{\beta\frac{2n+1}{2n+2}}
e^{- \frac{1}{4mR} \rho \log \rho + \frac{1}{2m} \log \rho - c_{0} \rho }.
$$
Choosing $ R $ large enough allows us to bound the sum over $ \ell $
and absorb it into the constant.

Going back to \eqref{f} where $ \partial_{t}^{\beta} w $ has been
replaced by $ \partial_{t}^{\beta} f_{2} $ we obtain
\begin{multline*}
\left| \int_{0}^{+\infty}  e^{iy\rho^{s_{0}}}
\rho^{r' + \alpha s_{0} + \frac{\beta s_{0}}{2(n+1)}}
(\partial_{t}^{\beta} f_{2}(t,\rho))_{\big|{
    t=\rho^{\frac{s_{0}}{2(n+1)}}x}} d\rho \right|
\\
\leq
C_{4}^{\beta+1} \beta^{\beta\frac{2n+1}{2n+2}} 
\int_{0}^{+\infty} \rho^{r' + \alpha s_{0} + \frac{\beta
    s_{0}}{2(n+1)}} e^{- \frac{1}{4mR} \rho \log \rho + \frac{1}{2m}
  \log \rho - c_{0} \rho } d\rho .
\end{multline*}
Applying Lemma \ref{L-log-Est}, for any $ \epsilon > 0 $, the right
hand side of the above relation is bounded by
$$ 
C_{4}^{\beta+1} \beta^{\beta\frac{2n+1}{2n+2}} \epsilon^{\alpha+
  \frac{\beta}{2n+2}} C_{\epsilon}
\left(\alpha+\frac{\beta}{2n+2}\right)^{\left(\alpha+\frac{\beta}{2n+2}\right)s_{0}} .
$$
With a slight change of notation we see that the above quantity is
estimated by
\begin{equation}
  \label{f2-fin}
C_{\epsilon} \epsilon^{\alpha + \beta} \alpha^{\alpha s_{0} }
\beta^{\beta s_{0}},
\end{equation}
since
$$ 
\frac{s_{0}}{2n+2} + \frac{2n+1}{2n+2} < s_{0}.
$$
Consider now 
$$ 
 \int_{0}^{+\infty}  e^{iy\rho^{s_{0}}}
\rho^{r' + \alpha s_{0} + \frac{\beta s_{0}}{2(n+1)}}
(\partial_{t}^{\beta} f_{1}(t,\rho))_{\big|{
    t=\rho^{\frac{s_{0}}{2(n+1)}}x}} d\rho  ,
$$
where $ f_{1} $ has been defined after \eqref{f}. Arguin as in
\eqref{A1} we may rewrite the integral above as
\begin{multline*}
\int_{0}^{+\infty}  e^{iy\rho^{s_{0}}}
\rho^{r' + \alpha s_{0} + \frac{\beta s_{0}}{2(n+1)}}
\Bigg[\partial_{t}^{\beta} \sum_{\ell \geq 1} \Big( \omega_{\ell} ( 1 - \Pi_{0}) \mathcal{P}_{0}
u_{\ell}
\\
+
\sum_{i=1}^{\min\{\ell, 2m\}}  \frac{1}{\rho^{i}}
\omega_{\ell-i} (1 - \Pi_{0}) \mathcal{P}_{i}(t, \partial_{t},
\partial_{\rho}) u_{\ell - i} \Big) \Bigg]_{\big|{
    t=\rho^{\frac{s_{0}}{2(n+1)}}x}} d\rho
\\
+
\int_{0}^{+\infty}  e^{iy\rho^{s_{0}}}
\rho^{r' + \alpha s_{0} + \frac{\beta s_{0}}{2(n+1)}}
\Bigg[\partial_{t}^{\beta} \sum_{\ell \geq 1} \Big(
\omega_{\ell}  \Pi_{0} \mathcal{P}_{0}
  u_{\ell} +
\frac{1}{\rho} \omega_{\ell} \Pi_{0} \mathcal{P}_{1} (1 - \Pi_{0})
u_{\ell}
\\
+
\sum_{i=2}^{\min\{\ell+1, 2m\}} \frac{1}{\rho^{i}}
\omega_{\ell+1-i} \Pi_{0} \mathcal{P}_{i}(t, \partial_{t},
\partial_{\rho}) u_{\ell +1 - i} \Big) \Bigg]_{\big|{
    t=\rho^{\frac{s_{0}}{2(n+1)}}x}} d\rho.
\end{multline*}
The quantities in square brackets
above have $ \rho $--support for $ 2R(\ell+1-2m) \leq \rho \leq
4R(\ell + 1) $.

The estimate of the derivatives of order $ \beta $ with respect to $ t
$ of the above terms is made according to the same lines of the
estimate above, since the order of $ \rho $-derivative of $
\omega_{\ell} $ played no role in the above estimate.

Hence we get \eqref{f2-fin} also in this case.
This ends the proof of the proposition.
\end{proof}

\section{Proof of Theorem \ref{Th}}
\renewcommand{\theequation}{\thesection.\arabic{equation}}
\setcounter{equation}{0} \setcounter{theorem}{0}

For the sake of completeness and to make
the paper self contained, we recall, in the Gevrey setting, the
following result due to M\'etivier.
\begin{theorem}[\cite{M80}]\label{Met_80_In}
	Let $P$ be a differential operator with analytic coefficient. Assume that in $\Omega$ open neighborhood of $x_{0} \in \mathbb{R}^{n}$
	there exist a continuous operator $R$ from $L^{2}(\Omega)$ in $L^{2}(\Omega)$ such that $PR= Id$. Then $P$ is $G^{s}$-hypoelliptic
	in $x_{0}$ if and only if for any neighborhood $\Omega_{2} \subset \Omega$ of $x_{0}$, there exists a neighborhood
	$\Omega_{3} \Subset \Omega_{2}$ of $x_{0}$ and constants $C$ and $L$ such that for any $k \in \mathbb{Z}_{+}$ and any $u \in \mathscr{D}'(\Omega_{2})$
	we have:
	\begin{itemize}
		\item[i)] $Pu \in H^{k}\left(\Omega_{2}\right) \Rightarrow u_{|_{\Omega_{3}}} \in H^{k}(\Omega_{3})$;
		\item[ii)] the following estimate holds 
		\begin{align}\label{Met_Est_Gs}
		\|u\|_{k,\Omega_{3}} \leq CL^{k} \left( \|\hspace*{-0.1em}| Pu|\hspace*{-0.1em}\|_{k,\Omega_{2}}+ k^{sk}\|u\|_{0,\Omega_{2}}\right),
		\end{align}
		 where
		 \begin{align*}
		 \|\hspace*{-0.1em}| Pu|\hspace*{-0.1em}\|_{k,\Omega_{2}} = \sum_{|\alpha| \leq k} k^{s(k-|\alpha|)}\|D^{\alpha} Pu\|_{0,\Omega_{2}}.
		 \end{align*}
	\end{itemize}  
\end{theorem}
We are going to use the above theorem to prove Theorem \ref{Th}.
Since the operator $M_{n,m}$ in (\ref{Op-GM}), as well as its adjoint,
satisfies the H\"ormander condition in $\Omega$, open neighborhood
of the origin, both $ M_{n,m} $ and its adjoint are subelliptic with a
loss of $2(1-1/(2n+1))$ derivatives (see \cite{H67} and \cite{RS}.)
As a consequence there is $\Omega_{1} \Subset \Omega$, neighborhood of the origin,
where  $M_{n,m}$ satisfies the assumption of Theorem \ref{Met_80_In}.

Arguing by contradiction,
assume that $M_{n,m}$ is $G^{s}$-hypoelliptic in a neighborhood of the
origin for $s <\frac{2m}{2m-1}$.

We showed before, in Proposition \ref{gamma}, that
$M_{n,m}\mathscr{K}[\tilde{u}] \in
\gamma^{2m/(2m-1)}_{g}(\mathbb{R}^{2})$;  
due to estimate (\ref{Met_Est_Gs}) and to the fact that
$G^{s_{1}}(\Omega_{2})  \subset \gamma^{s_{2}}(\Omega_{2})$  if
$s_{1}< s_{2}$ (see Lemma \ref{Ge_vs_Ga},)
we have that $ \mathscr{K}[\tilde{u}] \in \gamma^{2m/(2m-1)}(\Omega_{3})$,
$\Omega_{3} \Subset \Omega_{2}$ open neighborhood of the origin.

Next we need to prove the following
\begin{proposition}
\label{Fourier}
For any $ \epsilon > 0 $ there exists a $ C_{\epsilon} > 0 $ such that
for any $ \alpha \in \mathbb{N} $ we have
\begin{equation}
\label{y2daK}
\| \langle y \rangle^{2} \partial^{\alpha}_{y}
\mathscr{K}[\tilde{u}](0, y) \|_{L^{\infty}(\mathbb{R})} \leq
C_{\epsilon} \epsilon^{\alpha} \alpha!^{s_{0}} ,
\end{equation}
where
$$ 
s_{0} = \frac{2m}{2m-1}.
$$
\end{proposition}
Let us show first that Proposition \ref{Fourier} allows us to prove
Theorem \ref{Th}.
\begin{lemma}
\label{Fourier1}
For any $ \epsilon > 0 $ there exists a $ C_{\epsilon} > 0 $, $ B > 0 $, such that
\begin{equation}
\label{FK}
| \mathscr{F} \big(\mathscr{K}[\tilde{u}]\big) (0, \eta) | \leq C_{\epsilon}
e^{- B \left( \frac{|\eta|}{\epsilon}\right)^{\frac{1}{s_{0}}} } .
\end{equation}
\end{lemma}
\begin{proof}
Since
\begin{multline*}
| \mathscr{F} \big(\mathscr{K}[\tilde{u}]\big) (0, \eta) | \leq
\frac{1}{|\eta|^{\alpha}} \int_{\mathbb{R}}  |\langle y \rangle^{2} \partial^{\alpha}_{y}
\mathscr{K}[\tilde{u}](0, y) | \frac{1}{\langle y \rangle^{2}} dy
\\
\leq
C_{\epsilon} \epsilon^{\alpha} \alpha!^{s_{0}} |\eta|^{-\alpha}
\int_{\mathbb{R}} \frac{1}{\langle y \rangle^{2}} dy
\leq
C'_{\epsilon} \epsilon^{\alpha} \alpha!^{s_{0}} |\eta|^{-\alpha} .
\end{multline*}
Hence
$$ 
| \mathscr{F} \big(\mathscr{K}[\tilde{u}]\big) (0, \eta)
|^{\frac{1}{s_{0}}} \left(
  \frac{|\eta|}{2\epsilon}\right)^{\frac{\alpha}{s_{0}}}
\frac{1}{\alpha!} \leq C''_{\epsilon}
\left(\frac{1}{2}\right)^{\frac{\alpha}{s_{0}}} ,  
$$
so that, summing over $ \alpha $ we finally obtain
$$ 
| \mathscr{F} \big(\mathscr{K}[\tilde{u}]\big) (0, \eta)
|^{\frac{1}{s_{0}}}
e^{\left(\frac{|\eta|}{2\epsilon}\right)^{\frac{1}{s_{0}}}} \leq
C'''_{\epsilon} ,
$$
from which we conclude the proof of the lemma.
\end{proof}
We need also a bound from below of the same quantity of Lemma
\ref{Fourier1}.
\begin{lemma}
\label{Fourier2}
There are constants $ M > 0 $, $ \mu \in \mathbb{R} $ such that for $
\eta \geq 4 R $, $ R $ large enough, we have
\begin{equation}
\label{FK2}
| \mathscr{F} \big(\mathscr{K}[\tilde{u}]\big) (0, \eta) | \geq M
|\eta|^{\mu} e^{-c_{0} |\eta|^{\frac{1}{s_{0}}}} .
\end{equation}
\end{lemma}
\begin{proof}
Since
$$ 
\mathscr{K}[\tilde{u}](0, y) = \int_{0}^{+\infty} e^{i y \rho^{s_{0}}}
\rho^{r} \tilde{u}(0, \rho) d\rho = \frac{1}{s_{0}}  \int_{-\infty}^{+\infty} e^{i y \eta}
\eta^{\frac{r+1}{s_{0}} -1} \tilde{u}(0, \eta^{\frac{1}{s_{0}}}) d\eta ,
$$
we have that
$$ 
\mathscr{F} \big(\mathscr{K}[\tilde{u}]\big) (0, \eta) =
\frac{2\pi}{s_{0}} \eta^{\frac{r+1}{s_{0}} -1} \tilde{u}(0,
\eta^{\frac{1}{s_{0}}}) .
$$
Then, using Lemma \ref{Dul} we may write that
\begin{multline*}
| \mathscr{F} \big(\mathscr{K}[\tilde{u}]\big) (0, \eta) |
\geq
\frac{2\pi}{s_{0}} |\eta|^{\frac{\Re r+1}{s_{0}} -1} \Big(
\omega_{0}(\eta^{\frac{1}{s_{0}}}) v_{0}(0) e^{-c_{0}
  |\eta|^{\frac{1}{s_{0}}}}
- \sum_{\ell \geq 1} \omega_{\ell}(\eta^{\frac{1}{s_{0}}})
|u_{\ell}(0, \eta^{\frac{1}{s_{0}}}) | \Big)
\\
=
\frac{2\pi}{s_{0}} |\eta|^{\frac{\Re r+1}{s_{0}} -1} \Big(
\omega_{0}(\eta^{\frac{1}{s_{0}}}) v_{0}(0) e^{-c_{0}
  |\eta|^{\frac{1}{s_{0}}}}
- \sum_{\ell \geq 1} \omega_{\ell}(\eta^{\frac{1}{s_{0}}}) e^{-c_{0}
  |\eta|^{\frac{1}{s_{0}}}}
|u_{\ell}(0, \eta^{\frac{1}{s_{0}}}) | e^{c_{0}
  |\eta|^{\frac{1}{s_{0}}}} \Big)
\\
\geq
\frac{2\pi}{s_{0}} |\eta|^{\frac{\Re r+1}{s_{0}} -1} \Big(
\omega_{0}(\eta^{\frac{1}{s_{0}}}) v_{0}(0) e^{-c_{0}
  |\eta|^{\frac{1}{s_{0}}}}
-
e^{-c_{0} |\eta|^{\frac{1}{s_{0}}}}
\sum_{\ell \geq 1} \omega_{\ell}(\eta^{\frac{1}{s_{0}}}) 
C_{u}^{\ell+1} (\ell+1)^{\frac{\ell}{s_{0}}} |\eta|^{-
  \frac{\ell}{s_{0}}} 
\Big) .
\end{multline*}
On the other hand, on $ \supp \omega_{\ell} $, $
\eta^{\frac{1}{s_{0}}} \geq 2 R (\ell + 1) $, so that the above
inequality becomes, if 
\begin{multline*}
| \mathscr{F} \big(\mathscr{K}[\tilde{u}]\big) (0, \eta) |
\geq
\frac{2\pi}{s_{0}} |\eta|^{\frac{\Re r+1}{s_{0}} -1} \Big(
\omega_{0}(\eta^{\frac{1}{s_{0}}}) v_{0}(0) e^{-c_{0}
  |\eta|^{\frac{1}{s_{0}}}}
\\
-
e^{-c_{0} |\eta|^{\frac{1}{s_{0}}}}
\sum_{\ell \geq 1} \left( \frac{1}{2R}\right)^{\ell}
C_{u}^{\ell+1}  (\ell+1)^{\ell \left(\frac{1}{s_{0}}-1\right)}  
\Big)
\\
\geq
\frac{2\pi}{s_{0}} |\eta|^{\frac{\Re r+1}{s_{0}} -1} e^{-c_{0}
  |\eta|^{\frac{1}{s_{0}}}}  \Big(
\omega_{0}(\eta^{\frac{1}{s_{0}}}) v_{0}(0)
- \frac{C_{u}^{2}}{2R}
\sum_{\ell \geq 0} \left( \frac{C_{u}}{2R}\right)^{\ell}
  (\ell+1)^{(\ell+1) \left(\frac{1}{s_{0}}-1\right)}  
\Big)
\\
\geq
M |\eta|^{\frac{\Re r+1}{s_{0}} -1} e^{-c_{0}
  |\eta|^{\frac{1}{s_{0}}}} ,
\end{multline*}
provided $ |\eta|^{\frac{1}{s_{0}}} \geq 4 R $ and $ R $ is large
enough, so that $ v_{0}(0) - C_{u}^{2} (2R)^{-1} S > 0 $, where $ S $
denotes the sum of the series.

This ends the proof of the lemma.
\end{proof}
Lemma \ref{Fourier1} and Lemma \ref{Fourier2} for $ \eta $ large and $
\epsilon $ small give a contradiction. Hence the operator $ M_{n, m} $
is Gevrey $ \frac{2m}{2m-1} $ hypoelliptic and this value is optimal.

To finish the proof we have to show that \eqref{y2daK} holds, or, in
other words, that $ \mathscr{K}[\tilde{u}](0, y) $ belongs to the
global Beurling class $ \gamma_{g}^{\frac{2m}{2m-1}}(\mathbb{R}) $.
\begin{proof}[Proof of Proposition \ref{Fourier}]
Since $ \mathscr{K}[\tilde{u}] \in \gamma^{\frac{2m}{2m-1}}(\Omega) $,
where $ \Omega $ is a neighborhood of the origin in $ \mathbb{R}^{2}
$, we need only to show that $ \mathscr{K}[\tilde{u}](0, y) \in
\gamma_{g}^{\frac{2m}{2m-1}}(\{ |y| \geq \delta\}) $, for a suitable
positive $ \delta $.

Actually we are going to prove \eqref{y2daK} for $ |y| \geq \delta $.

For $ \alpha \in \mathbb{N} $, consider
$$ 
y^{2} \partial_{y}^{\alpha} \mathscr{K}[\tilde{u}](0, y) .
$$
If the factor $ y^{2} $ is missing the argument is the same, hence we
skip it.

Let us compute the above quantity:
$$
y^{2} D_{y}^{\alpha}  \int_{0}^{+\infty} e^{i y \rho^{s_{0}}}
\rho^{r} \tilde{u}(0, \rho) d\rho = y^{2} \int_{0}^{+\infty} e^{i y \rho^{s_{0}}}
\rho^{r + \alpha s_{0}} \tilde{u}(0, \rho) d\rho.
$$
Now
$$ 
e^{i y \rho^{s_{0}}} = \frac{1}{y s_{0} \rho^{s_{0}-1}} D_{\rho} e^{i
  y \rho^{s_{0}}} ,
$$
so that
$$ 
y^{2} D_{y}^{\alpha}  \int_{0}^{+\infty} e^{i y \rho^{s_{0}}}
\rho^{r} \tilde{u}(0, \rho) d\rho =
 y^{2} \int_{0}^{+\infty} e^{i y \rho^{s_{0}}}  \left(- D_{\rho}
   \frac{1}{y s_{0} \rho^{s_{0}-1}} \right)^{\beta}
\rho^{r + \alpha s_{0}} \tilde{u}(0, \rho) d\rho,
$$
where we made $ \beta $ integrations by parts using the fact that $
\tilde{u}(0, \rho) $ is rapidly vanishing at infinity.

As for the factor $ y^{2} $ we transform it into derivatives with
respect to $ \rho $, finally obtaining
\begin{multline*}
y^{2} D_{y}^{\alpha}  \int_{0}^{+\infty} e^{i y \rho^{s_{0}}}
\rho^{r} \tilde{u}(0, \rho) d\rho
\\
=
\int_{0}^{+\infty} e^{i y \rho^{s_{0}}}
\left(- D_{\rho}
   \frac{1}{ s_{0} \rho^{s_{0}-1}} \right)^{2}
\left(- D_{\rho}
   \frac{1}{y s_{0} \rho^{s_{0}-1}} \right)^{\beta}
\rho^{r + \alpha s_{0}} \tilde{u}(0, \rho) d\rho
\\
=
\frac{1}{ y^{\beta} } \int_{0}^{+\infty} e^{i y \rho^{s_{0}}}
 \left( - \partial_{\rho}
   \frac{1}{i s_{0} \rho^{s_{0}-1}} \right)^{\beta+2}
\rho^{r + \alpha s_{0}} \tilde{u}(0, \rho) d\rho ,
\end{multline*}
where $ a $, $ b $, are suitable constants. We use the formula
\begin{equation}
\label{ff}
\left( - \partial_{\rho} \frac{1}{i s_{0} \rho^{s_{0}-1}}
\right)^{\beta+2} = \sum_{h=0}^{\beta+2} \gamma_{\beta+2, h}
\frac{1}{\rho^{s_{0} (\beta+2) - h}} \partial_{\rho}^{h},
\end{equation}
where
\begin{equation}
\label{ff1}
|\gamma_{\beta+2, h} | \leq C_{\gamma}'^{\beta+2+h} \frac{(\beta+2)!}{h!} \leq
C_{\gamma}^{\beta+2+h} (\beta+2 - h)!.
\end{equation}
Here $ C'_{\gamma} $, $ C_{\gamma} $ are positive constants
independent of $ \beta $, $ h $. In particular we have $ \gamma_{\beta+2,
\beta+2} = \left( \frac{i}{s_{0}}\right)^{\beta+2} $, and for convenience
we set $ \gamma_{0 0} = 1 $. Thus the above integral becomes
$$ 
\frac{1}{ y^{\beta} } \sum_{h=0}^{\beta+2} \gamma_{\beta+2, h}  \int_{0}^{+\infty} e^{i y \rho^{s_{0}}}
\frac{1}{\rho^{s_{0} (\beta+2) - h}} \partial_{\rho}^{h} \left(
\rho^{r + \alpha s_{0}} \tilde{u}(0, \rho)\right)  d\rho . 
$$
We use the formula
\begin{equation}
\label{(.)_} 
\partial_{\rho}^{p} \left( \rho^{\lambda} u \right) =
\sum_{k = 0}^{p} \binom{p}{k}
(\lambda)_{p-k} \rho^{\lambda - p + k} \partial_{\rho}^{k} u,
\end{equation}
where $ (\lambda)_{\beta} $ is the Pochhammer symbol defined by
\begin{equation}
\label{eq:lambdabeta} 
(\lambda)_{\beta} = \lambda (\lambda -1) \cdots (\lambda - \beta + 1),
\quad (\lambda)_{0} = 1,
\quad \lambda \in \mathbb{C}.
\end{equation}
Then the above integral becomes
\begin{equation}
  \label{int1}
\frac{1}{ y^{\beta} } \sum_{h=0}^{\beta+2} \sum_{k = 0}^{h} \binom{h}{k}
(r+\alpha s_{0})_{h-k} \gamma_{\beta+2, h}  \int_{0}^{+\infty} e^{i y \rho^{s_{0}}}
\rho^{r+\alpha s_{0} + k -s_{0}(\beta+2)} \partial_{\rho}^{k}
\tilde{u}(0, \rho)  d\rho  . 
\end{equation}
Consider now
$$ 
\partial_{\rho}^{k} \tilde{u}(0, \rho) = \sum_{\ell \geq 0}
\partial_{\rho}^{k} 
\left(u_{\ell}(0, \rho) \omega_{\ell}(\rho) \right) = \sum_{\ell \geq
  0} \sum_{i=0}^{k} \binom{k}{i} \partial_{\rho}^{i} u_{\ell}(0, \rho)
\partial_{\rho}^{k-i} \omega_{\ell}(\rho).
$$
by \eqref{utilde}. The absolute value of the quantity in \eqref{int1}
for $ |y| \geq \delta $ can be estimated, applying Lemma
\ref{lemma:cutoff} and Lemma \ref{Dul}, by
\begin{multline*}
  \delta^{-\beta} \sum_{h=0}^{\beta+2} \sum_{k = 0}^{h} \sum_{\ell \geq
  0} \sum_{i=0}^{k} \binom{k}{i} \binom{h}{k}
| (r+\alpha s_{0})_{h-k}|  |\gamma_{\beta+2, h}|
\\
\cdot
\int_{0}^{+\infty}
  \rho^{r+\alpha s_{0} + k -s_{0}(\beta+2)} | \partial_{\rho}^{i}
  u_{\ell}(0, \rho) | |\partial_{\rho}^{k-i} \omega_{\ell}(\rho)| d
  \rho
\\
\leq
\delta^{-\beta} \sum_{h=0}^{\beta+2} \sum_{k = 0}^{h} \sum_{\ell \geq
  0} \sum_{i=0}^{k} C_{\gamma}'^{\beta+2-h} \binom{k}{i} \binom{h}{k}
| (r+\alpha s_{0})_{h-k}| \frac{(\beta+2)!}{h!}
\\
\cdot
(R C_{\omega})^{k-i+1} (k-i)!^{\sigma} 
C_{u}^{\ell+i+1} (\ell+1)^{\frac{\ell}{s_{0}} + \frac{i}{2m}}
\int_{2 R(\ell+1)}^{+\infty} \rho^{r+\alpha s_{0} + i -s_{0}(\beta+2) -\ell }
e^{-c_{0} \rho} d\rho 
\end{multline*}
Choose $ \beta + 2 = \alpha $. There is no problem in assuming that $
\alpha > 2 $, since we are interested in large values of $ \alpha
$. Furthermore on the domain of integration we have that
$$ 
\frac{1}{\rho} \leq \frac{1}{2 R (\ell+1)},
$$
so that the quantity on the right hand side of the above inequality is
bounded by
\begin{multline*}
\delta^{-\alpha+2} C^{\alpha} \sum_{h=0}^{\alpha} \sum_{k = 0}^{h} \sum_{\ell \geq
  0}  \sum_{i=0}^{k} C^{-h+k+1}  \left(\frac{C}{R}\right)^{\ell} \frac{\alpha!}{i!}
\frac{| (r+\alpha s_{0})_{h-k}|}{(h-k)!} (k-i)!^{\sigma-1} 
\\
\cdot
R^{k-i+1 - \frac{i}{2m}}  (\ell+1)^{\ell\left( \frac{1}{s_{0}} -1\right)}
\int_{2 R(\ell+1)}^{+\infty} \rho^{r + i \left( 1 + \frac{1}{2m}\right)}
e^{-c_{0} \rho} d\rho .
\end{multline*}
The sum over $ \ell \geq 0 $ yields a constant independent of $ \alpha
$:
\begin{multline*}
\delta^{-\alpha+2} C^{\alpha+1} \sum_{h=0}^{\alpha} \sum_{k = 0}^{h}  \sum_{i=0}^{k} C^{-h+k+1} \frac{\alpha!}{i!}
\frac{| (r+\alpha s_{0})_{h-k}|}{(h-k)!} (k-i)!^{\sigma-1} 
\\
\cdot
R^{k-i+1 - \frac{i}{2m}} 
\int_{0}^{+\infty} \rho^{r + i \left( 1 + \frac{1}{2m}\right)}
e^{-c_{0} \rho} d\rho .
\end{multline*}
Now
\begin{multline*}
\frac{| (r+\alpha s_{0})_{h-k}|}{(h-k)!} =
\frac{|r+\alpha s_{0}| \ |r+\alpha s_{0} - 1| \ldots | r+\alpha s_{0}
  - h+k+1| }{(h-k)!}
\\
\leq
\frac{ (|r| + \alpha s_{0} ) (|r| + \alpha s_{0} -1) \ldots (|r| +
  \alpha s_{0} - (h-k) + 1)}{(h-k)!}
\\
=
s_{0}^{h-k} \frac{\left( \frac{|r|}{s_{0}} + \alpha\right) \left(
    \frac{|r|}{s_{0}} + \alpha - \frac{1}{s_{0}}\right) \ldots
 \left( \frac{|r|}{s_{0}} + \alpha - \frac{h - k + 1}{s_{0}}\right)
}{(h-k)!}
\leq
s_{0}^{h-k} \frac{\left( \frac{|r|}{s_{0}} + \alpha
  \right)^{h-k}}{(h-k)!}
\\
\leq
C_{0}^{h-k} \frac{\alpha^{h-k}}{(h-k)!} = C_{0}^{h-k}  \frac{\alpha!}{(\alpha -
  (h-k))!} \frac{\alpha^{h-k}}{\alpha!} \frac{(\alpha -
  (h-k))!}{(h-k)!}
\\
=
C_{0}^{h-k}
\binom{\alpha}{h-k} (\alpha - (h-k))! \frac{\alpha^{h-k}}{\alpha!}
\leq
C_{1}^{\alpha} \frac{(\alpha - (h-k))!}{\alpha^{\alpha - (h-k)}} \leq
C_{2}^{\alpha} . 
\end{multline*}
Here we assumed, without loss of generality, that $ \alpha \geq
\frac{|r|}{s_{0}} $ and we also used the estimate $ n! \geq C^{n}
n^{n} $. Plugging this into the above expression we get
\begin{multline*}
\delta^{-\alpha+2} C^{\alpha+1} \sum_{h=0}^{\alpha} \sum_{k = 0}^{h}  \sum_{i=0}^{k} C^{-h+k+1} \frac{\alpha!}{i!} (k-i)!^{\sigma-1} 
\\
\cdot
R^{k-i+1 - \frac{i}{2m}} 
\int_{0}^{+\infty} \rho^{r + i \left( 1 + \frac{1}{2m}\right)}
e^{-c_{0} \rho} d\rho ,
\end{multline*}
with a different meaning of the constant involved. As a consequence
the above quantity can be further estimated as
$$ 
\delta^{-\alpha+2} C^{\alpha+1} \sum_{h=0}^{\alpha} \sum_{k = 0}^{h}
\sum_{i=0}^{k} C^{-h+k+1} \frac{\alpha!}{i!} (k-i)!^{\sigma-1}
R^{k-i+1 - \frac{i}{2m}} M^{i} i!^{1+\frac{1}{2m}} . 
$$
If we choose $ \sigma $ such that $ \sigma - 1 = \frac{1}{2m} $, we
have
$$ 
\delta^{-\alpha+2} C^{\alpha+1} \sum_{h=0}^{\alpha} \sum_{k = 0}^{h}
\sum_{i=0}^{k} C^{-h+k+1} \alpha! 
R^{k-i+1 - \frac{i}{2m}} M^{i} k!^{\frac{1}{2m}} \leq
\delta^{-\alpha+2} C_{1}^{\alpha+1} \alpha!^{1 + \frac{1}{2m}} . 
$$
Hence for $ |y| \geq \delta $ we have that $ \mathscr{K}[\tilde{u}](0,
y) $ belongs to $ G^{1 + \frac{1}{2m}}(\{ |y| \geq \delta\}) $ and,
since
$$ 
1 + \frac{1}{2m} < \frac{2m}{2m - 1} = s_{0},
$$
we proved that $ \mathscr{K}[\tilde{u}](0, y) $ belongs to $
\gamma_{g}^{\frac{2m}{2m-1}}(\mathbb{R}) $ and, moreover, each
derivative is $ L^{1} $ summable in the variable $ y $.

\end{proof}
%

%
%
%

\section{Appendix}
%
%
%
In this appendix we collect some results on the eigenfunctions of the
Bender and Wang operator. We also include here some estimates we use in the
preceding sections.

\subsection{On a special eigenvalue problem}
\renewcommand{\theequation}{\thesubsection.\arabic{equation}}
\renewcommand{\thetheorem}{\thesubsection.\arabic{theorem}}
\renewcommand{\theproposition}{\thesubsection.\arabic{proposition}}
\renewcommand{\thelemma}{\thesubsection.\arabic{lemma}}
\renewcommand{\theremark}{\thesubsection.\arabic{remark}}
\renewcommand{\thedefinition}{\thesubsection.\arabic{definition}}
\setcounter{equation}{0} \setcounter{theorem}{0}
In \cite{Bender2001}, Bender and Wang study the eigenvalue problem
\begin{align*}
\left(-\partial_{t}^{2} + t^{2N+2}\right) u(t)= E t^{N}u(t), \quad N=-1,0,1,2,\dots,
\end{align*}
on the interval $-\infty \leq t \leq +\infty$. The eigenfunction $u(t)$ is required to obey the boundary
conditions that $u(t)$ vanish exponentially rapidly as $t\rightarrow \pm \infty$.\\
In this paragraph we focus on the case $N=2n+1$, $n \in \mathbb{Z}_{+}$,  
\begin{align}\label{Eig_P}
\left(-\partial_{t}^{2} + t^{2(2n+1)}\right) u(t)= E t^{2n}u(t).
\end{align}
In \cite{Bender2001} the authors show that the even-parity eigenfunctions have the form
\begin{align}\label{Ev_Eigf}
v_{k}(t) = e^{-\frac{t^{2n+2}}{2n+2}} L^{(-1/(2n+2))}_{k}\left(\frac{t^{2n+2}}{n+1}\right), \quad k\in \mathbb{Z}_{+}, 
\end{align}
%
%
with corresponding eigenvalue 
\begin{align*}
E_{k} = 4k(n+1)+2n+1;
\end{align*} 
and the odd-parity eigenfunctions have the form
\begin{align}\label{Odd_Eigf}
w_{k}(t) = e^{-\frac{t^{2n+2}}{2n+2}} \, t \,L^{(1/(2n+2))}_{k}\left(\frac{t^{2n+2}}{n+1}\right),\quad k\in \mathbb{Z}_{+} 
\end{align}
%
%
with corresponding eigenvalue 
\begin{align*}
\tilde{E}_{k}= 4k(n+1)+2n+3.
\end{align*} 
$L_{k}^{(\alpha)}(\cdot)$, $\alpha = \pm1/(2n+2)$, are the generalized \textit{Laguerre polynomials} given by
\begin{align*}
L^{(\alpha)}_{k}\left(\frac{t^{2n+2}}{n+1}\right)=\sum_{i=0}^{k} (-1)^{k} \binom{k+\alpha}{k-i} \frac{t^{2i(n+1)}}{(n+1)^{i}i!}\,.
\end{align*}
In the next Proposition we list some properties of the
\textit{Laguerre polynomials} that we use in what follows:
\begin{proposition}\label{Pr_Lag_Pol}
	The following properties concerning the generalized Laguerre polynomials hold:
	\begin{enumerate}
		\item[(L-1)]  let $k\in \mathbb{Z}_{+}$ and $\alpha\geq -1/2$ then
		\begin{align}\label{Lag_p_1}
		(k+1) L_{k+1}^{\alpha} (s) - \left(2k+1-\alpha -s\right)L_{k}^{\alpha}(s)+(k+\alpha)L_{k-1}^{\alpha}(s)=0,
		\end{align}
		and
		\begin{align}\label{Lag_p_2}
		s\frac{d}{ds} L_{k}^{\alpha}(s) = kL_{k}^{\alpha}(s) - (k+\alpha)L_{k-1}^{\alpha}(s);
		\end{align}
		\item[(L-2)]  let $k, q \in \mathbb{Z}_{+}$ and $\alpha > -1$ then
		\begin{align}\label{Orth_Lag_p}
		\int_{0}^{+\infty} e^{-s} s^{\alpha} L_{k}^{\alpha}(s) L_{q}^{\alpha}(s) ds = \Gamma(\alpha +1) \binom{k+\alpha}{k} \delta_{k,q}, 
		\end{align}
		where $\delta_{k,q}$ is the Kroneker symbol;
		\item[(L-3)] due to the Theorem 6.23, \cite{Szego}, if $\alpha > -1$ and $k\in \mathbb{Z}_{+}$ then $L_{k}^{\alpha}(s)$ has $k$ positive zeros.
	\end{enumerate}
\end{proposition}
\begin{remark}\label{Part_Orth}
	Due to the property (L-3), Proposition \ref{Pr_Lag_Pol}, we have that $v_{k}(t)$, (\ref{Ev_Eigf}), has  $2k$ zeros each one of multiplicity $2n+2$ and $w_{k}(t)$, (\ref{Ev_Eigf}), has  $2k+1$ zeros, $2k$ of them with multiplicity $2n + 2$, $0$ is a zero.\\
	From the property (L-2), Proposition \ref{Pr_Lag_Pol}, we obtain
	\begin{align}\label{Orth-v_k}
	\|t^{n}v_{k}(t)\|^{2}_{0}= \frac{1}{2} \left(\frac{1}{n+1}\right)^{\frac{1}{2n+2}} \Gamma(1-\alpha) \binom{k-\alpha}{k} 
	\end{align}
	and
	\begin{align}\label{Orth-w_k}
	\|t^{n}w_{k}(t)\|^{2}_{0}= \frac{1}{2}  \left(\frac{1}{n+1}\right)^{\frac{1}{2n+2}} \Gamma(1+\alpha) \binom{k+\alpha}{k} 
	\end{align}
	where $\alpha =(2n+2)^{-1}$.
        
	Moreover we have $\langle t^{n}v_{k}, t^{n}v_{m} \rangle =0$ and  $\langle t^{n}w_{k}, t^{n}w_{m} \rangle =0$
	for any $k,m\in \mathbb{Z}_{+}$ with $k\neq m$. From now on we
        shall consider the eigenfunctions, $ v_{k} $, $ w_{k} $, as
        normalized with respect the norms in \eqref{Orth-v_k},
        \eqref{Orth-w_k}. 
\end{remark}
The next lemma gives us a three terms recurrence relation for some
derivative of the eigenfunctions. This relation is crucial for our
computation of the asymptotic solution.
\begin{lemma}\label{Rk-Ric-Rel}
	The following recurrence relation holds
	\begin{align}\label{Rec_Lag_1}
	t\frac{d}{dt} v_{k}(t) = (n+1) \Big(  (k+1) v_{k+1}(t)\! -
          (1+\alpha) v_{k}(t)\! - (k+\alpha) v_{k-1}(t) \Big), \, k\geq 1,
	\end{align}
	and 
	\begin{align}\label{Rec_Lag_0}
	t\frac{d}{dt} v_{0}(t) = (n+1)\left( v_{1}(t)  - (1+\alpha) v_{0}(t)\right),
	\end{align}
	where $\alpha = -\frac{1}{2n+2}$. More in general for every $i,k \in \mathbb{Z}_{+}$ we have
	\begin{align}\label{Rec_Lag_2}
	\left(t\frac{d}{dt}\right)^{i} v_{k}(t) =\sum_{j=-i}^{i} \delta^{k,i}_{j} v_{k+j}(t) = 
	\sum_{j =\max\lbrace k-i,0\rbrace}^{k+i} \delta^{k,i}_{j-k}v_{j}(t).
	\end{align}
	where:
	\begin{enumerate}
			\item[(i)] $
                          \delta_{i}^{k,i}=(n+1)^{i}\frac{(k+i)!}{k!}
                          $, and, more generally,
			\item[(ii)] $|\delta_{j}^{k,i}| \leq C^{i}
                          \frac{(k+i)!}{k!}$, $j= -i, \dots, i$, where
                          $C$ is a suitable positive constant.
	\end{enumerate}
In the first sum in \eqref{Rec_Lag_2} we used the convention that $v_{\ell}(t)$ is
identically zero if $\ell$ is negative.
\end{lemma}
\begin{proof}
	Taking advantage of the relations (\ref{Lag_p_1}) and (\ref{Lag_p_2}) in (L-1), Proposition \ref{Pr_Lag_Pol}, it easy to obtain the recurrence relation
	(\ref{Rec_Lag_1}) and (\ref{Rec_Lag_0}).

        Setting $\delta^{k,i}_{j} = (n+1)^{i}
        \tilde{\delta}^{k,i}_{j}$ a direct computation allows us to
        obtain recursively the coefficients in the sum (\ref{Rec_Lag_2}):
	\begin{enumerate}
		\item[$\bullet$] $\tilde{\delta}_{i}^{k,i}= (k+i)\tilde{\delta}_{i-1}^{k,i-1}$;
		\item [$\bullet$]$\tilde{\delta}_{i-1}^{k,i}= (k+i-1)\tilde{\delta}_{i-2}^{k,i-1}- (1+\alpha) \tilde{\delta}_{i-1}^{k,i-1}$;
		\item [$\bullet$]$\tilde{\delta}_{j}^{k,i}= (k+j)\tilde{\delta}_{j-1}^{k,i-1}- (1+\alpha) \tilde{\delta}_{j}^{k,i-1}-(k+j+1+\alpha)\tilde{\delta}_{j+1}^{k,i-1}$, $j=-i+2, \dots, i-2$;
		\item [$\bullet$]$\tilde{\delta}_{-i+1}^{k,i}= (k+\alpha-i+2)\tilde{\delta}_{-i+2}^{k,i-1}- (1+\alpha) \tilde{\delta}_{-i+1}^{k,i-1}$;
		\item [$\bullet$]$\tilde{\delta}_{-i}^{k,i}= (k+\alpha-i+1)\tilde{\delta}_{-i+1}^{k,i-1}$;
	\end{enumerate}
	where $\tilde{\delta}^{k,i-1}_{\ell}$, $ \ell \in \lbrace i-1,
        -i+1 \rbrace$, are the coefficients of
        $(n+1)^{-i+1}\left(t\frac{d}{dt}\right)^{i-1} v_{k}(t)$.
   	From the above relations and arguing by induction we obtain $ \tilde{\delta}_{i}^{k,i}=\frac{(k+i)!}{k!} $, $|\tilde{\delta}_{i-1}^{k,i}|  \leq   i (1+\alpha) \frac{(k+i-1)!}{k!}$, $\tilde{\delta}_{-i}^{k,i}= \prod_{\ell = 0}^{i-1}(k+\alpha - \ell)$ and
   $|\tilde{\delta}_{j}^{k,i}| \leq C^{i} \frac{(k+i)!}{k!}$, $j= -i+1, \dots, i-2$, where $C$ is a suitable positive constant.
\end{proof}	
\begin{remark}
	From the above Lemma and taking advantage of relation
        (\ref{Orth-v_k}), Remark \ref{Part_Orth}, we can describe the
	action of the operators  $\mathscr{P}_{i}(t\partial_{t})$,
        (\ref{OP-Pi}), on the even-parity eigenfunctions. We have
	\begin{enumerate}
		\item for $p\leq i$
		\begin{align}\label{P_i_C1}
		\mathscr{P}_{i}(t\partial_{t})v_{p}(t)=
                  \sum_{\nu=0}^{i+p} \left(\sum_{j=|\nu-p|}^{i}
                  \text{\textcursive{ p}}_{i,j}\delta_{\nu - p}^{p,j} \right)v_{\nu}(t);
		\end{align}
		\item for $p\geq i$
		\begin{align}\label{P_i_C2}
		\mathscr{P}_{i}(t\partial_{t})v_{p}(t)=
                  \sum_{\nu=p-i}^{p+i}
                  \left(\sum_{j=|\nu-p|}^{i}\!\!\!\text{\textcursive{p}}_{i,j}\delta_{\nu
                  - p}^{p,j} \right)v_{\nu}(t). 
		\end{align}
	\end{enumerate} 
\end{remark}
\vspace{1em}
Next we show that the  eigenfunctions $v_{k}(t)$ belong to the  \textit{Gel'fand-Shilov space} $S^{(2n+1)/(2n+2)}_{1/(2n+2)}(\mathbb{R})$;
the same approach can be used in the case of the odd-parity
eigenfunctions $w_{k}(t)$.
\begin{definition}
	Let $\alpha$ and $\beta$ be real positive numbers. By $S^{\alpha}_{\beta}\left(\mathbb{R}\right)$ we denote the set of infinitely differentiable functions $f(t)$,
	in $\mathbb{R}$, satisfying the inequality
	\begin{align}\label{GS-sp}
	|t^{k}f^{(q)}(t)| \leq C A^{k}B^{q} k^{\alpha k} q^{\beta q}
	\end{align} 
	where the positive constants $C$, $A$ and $B$ depend only on $f(t)$.
\end{definition}
Fore more details on the subject we refer to \cite{Gelfand_Shilov}.

In order to prove the following proposition we follow the ideas of Gundersen, \cite{Gundersen_78}, and of Titchmarsh, \cite{Titich_P1}.
\begin{proposition}\label{L-inf-v_k}
	There exists a positive constant $ C_{0} $, such that the
        following estimates hold
	\begin{align}\label{Est_Eig}
	\|v_{k}(t)\|_{\infty} 
          \leq   
        C_{0}  E_{k}^{\frac{3}{2}+\frac{1}{4n+4}}
	\quad \text{ and } \quad
	\|w_{k}(t)\|_{\infty} 
          \leq  
         C_{0} E_{k}^{\frac{3}{2}+\frac{1}{4n+4}}.
	\end{align}
\end{proposition}
\begin{proof}
We set $Q_{k}(t) \doteq t^{2n}\left( t^{2n+ 2} - E_{k}\right)$. Even
though $ Q_{k} $ is not a polynomial of the form considered in
\cite{Gundersen_78}, \cite{Titich_P1}, we observe that it differs from
a polynomial of that form by the factor $ t^{2n} $. Hence if we work
on the complement of a small interval centered at the origin, we may
argue along the same lines of \cite{Gundersen_78}, \cite{Titich_P1},
Sections 5.4 and 8.4.1.

In what follows we work for $ t > 0 $. A symmetric argument holds for
$ t < 0 $.

Set $T_{k} \doteq E_{k}^{1/(2n+2)}$, so that $Q_{k}(T_{k})=0$ and $
Q_{k}(t) \neq 0 $ if $ t > T_{k} $.

Since $v_{k}(t)$ and $v''_{k}(t)$ have the same sign for $ t > T_{k}$,
all the zeros of $v_{k}(t)$ are in the interval $(0, T_{k})$.

Let $t_{1}$ be a zero of $v_{k}(t)$, so that $v''_{k}(t_{1}) = 0$, and
let $t_{2} > t_{1} $ be the next critical point of $v_{k}(t)$. 
To be definite we assume that $t_{2}$ is a maximum, the case of a
minimum being treated in the same way.

Let $s$ be a point in the interval $[t_{1}, t_{2}]$. Since $v''_{k}(t)
= -Q_{k}(t)v_{k}(t)$ we have 
	\begin{align*}
	\left| v_{k}(t_{2})\right| 
	&=\left|  \int_{t_{1}}^{t_{2}}v'_{k}(s)\,ds \right| = \left|\int_{t_{1}}^{t_{2}} \int_{s}^{t_{2}} t^{2n}\left(E_{k}-t^{2n+2}\right)v_{k}(t)\, dt\,ds\right|\\
	&\leq E_{k}\left| \int_{t_{1}}^{t_{2}}\int_{s}^{t_{2}} t^{2n}v_{k}(t)\, dt\,ds \right| 
	\leq  E_{k} |T_{k}|^{n}\int_{t_{1}}^{t_{2}}
   \left(t_{2}-s\right)^{1/2}\|t^{n}v_{k}(t)\|_{L^{2}(\mathbb{R})} \,
   ds \\
	&\leq \frac{2}{3} E_{k}^{1+\frac{n}{2n+2}} \left(t_{2}-t_{1}\right)^{3/2}
	\leq  \frac{2}{3} E_{k}^{1+\frac{n}{2n+2}} T_{k}^{3/2}.
	\end{align*}
Same argument if $ t_{2} $ is a minimum. As a consequence
the first of the estimates in \eqref{Est_Eig} follows from the
        fact that both $ v_{k} $ is rapidly decreasing
        at infinity.
        
	The same argument gives the second estimate in (\ref{Est_Eig}).
\end{proof}
Proposition \ref{L-inf-v_k} allows us to estimate how the $ L^{\infty}
$ norm of $ v_{k} $ depends on $ k $.
\begin{corollary}
\label{k-dep}
There is a positive constant, $ C_{v} $, such that
\begin{equation}
\label{k-dep-1}
\| v_{k} \|_{\infty} \leq C_{v} (k+1)^{\frac{3}{2} + \frac{1}{4n+4}} .
\end{equation}
\end{corollary}
\begin{proposition}
	There exists a positive constant $ C_{1} $, such that the
        following estimates hold
	\begin{align}\label{Est_Der_Eig}
	\|v'_{k}(t)\|_{\infty} 
          \leq   
        C_{1} E_{k}^{\frac{7}{2}-\frac{1}{4n+4}}
	\quad \text{ and } \quad
	\|w'_{k}(t)\|_{\infty} 
          \leq   
        C_{1} E_{k}^{\frac{7}{2}-\frac{1}{4n+4}}.
	\end{align}
\end{proposition}
\begin{proof}
	Let $t_{0}$ be such that $|t_{0}| < T_{k}$ and let $t_{1}$
        denote the closest critical point for $v_{k}(t)$. Then
	\begin{align*}
	|v'_{k}(t)|& \leq \left| \int_{t_{1}}^{t_{0}} \!\!\! v''_{k}(s) \, ds \right| =  \left| \int_{t_{1}}^{t_{0}}\!\!\! s^{2n}\left(s^{2n+2}-E_{k}\right) v_{k}(s) \, ds \right| \\
	&\leq E_{k} T_{k}^{2n} \|v_{k}\|_{\infty} |t_{0}-t_{1}|
	\leq E_{k} T_{k}^{2n+1} \|v_{k}\|_{\infty}. 
	\end{align*}
	On the other hand, since $ |v_{k}(t)| \rightarrow 0$  and
        $|v'_{k} (t)| \rightarrow 0$ for $t \rightarrow \pm \infty$ we
        obtain the first of the estimates in \eqref{Est_Der_Eig}.
        
	Same argument for the second estimate in (\ref{Est_Der_Eig}). 	
\end{proof}
\begin{proposition}\label{Exp-Est_Eig}
	The following estimate holds
	\begin{align}
          \label{vk-est-real}
	|v_{k}(t)| \leq C_{0}^{k+1}  e^{-Bt^{2n+2}},
	\end{align}
        where
        $$
        B=  \frac{1-\delta^{2n+1}}{2n+2}  
        $$
        where $\delta \leq 1/2$ and $ C_{0}$ is a suitable constant
        independent of $ k $.
\end{proposition}
\begin{proof}
	Let $t_{0}$ be such that $Q_{k}(t_{0})$ and $v_{k}(t_{0})$ are both positive.
	(In the odd-parity case, $w_{k}(t)$, we have to consider $|w_{k}(t_{0})|$ since for $t_{0}< -T_{k} $, $Q_{k}(t_{0})> 0$ and $w_{k}(t_{0}) <0$.)

        We assume that $t_{0}> T_{k}$. We have that for $t> t_{0}$, $Q_{k}(t)> 0$, $v_{k}(t) > 0$ and $v'_{k}(t) < 0$, $v_{k}(t) \rightarrow 0 $ as $t \rightarrow + \infty$ 
	(in the odd parity case $w_{k}(t) < 0$ and $w'_{k}(t) > 0$ as $t \rightarrow - \infty$.)
	We have 
	\begin{align*}
	- v'_{k}(t) v''_{k}(t) = Q_{k}(t) v_{k}(t) (-v'_{k}(t)) \geq Q_{k}(t_{0}) v_{k}(t) (-v'_{k}(t)).  
	\end{align*}  
	%
	We have
	\begin{align*}
	\int_{t_{0}}^{t_{1}} v'_{k}(t) v''_{k}(t) dt
	=\left. \frac{\left(v'_{k}(t)\right)^{2}}{2}\right|_{t_{0}}^{t_{1}}
	\quad \text{ and } \quad
	\int_{t_{0}}^{t_{1}} v_{k}(t) v'_{k}(t) dt
	=\left. \frac{\left(v_{k}(t)\right)^{2}}{2}\right|_{t_{0}}^{t_{1}}.
	\end{align*} 
	We obtain 
	\begin{align*}
	-\frac{\left(v'_{k}(t_{1})\right)^{2}}{2} + \frac{\left(v'_{k}(t_{0})\right)^{2}}{2}
	\geq Q_{k}(t_{0}) \left[-\frac{\left(v_{k}(t_{1})\right)^{2}}{2}  + \frac{\left(v_{k}(t_{0})\right)^{2}}{2}\right].
	\end{align*}
	Taking $t_{1} \rightarrow + \infty $, since both
        $v_{k}(t_{1}) \rightarrow 0$ and  $v'_{k}(t_{1}) \rightarrow 0$, we have
	\begin{align*}
	\left(v'_{k}(t_{0})\right)^{2} \geq Q_{k}(t_{0}) \left(v_{k}(t_{0})\right)^{2}
	\text{ or equivalently } \left(- \frac{v'_{k}(t_{0})}{v_{k}(t_{0})}\right)^{2} \geq Q_{k}(t_{0}).
	\end{align*}
	Without loss of generality we replace $t_{0}$ by $t$, $t >  T_{k}$. Integrating  both side on the interval $(t_{0}, t)$ we have
	\begin{align*}
	- \int_{t_{0}}^{t}  \frac{v'_{k}(s)}{v_{k}(s)} \, ds = \log v_{k}(t_{0}) - \log v_{k}(t) > \int_{t_{0}}^{t} \left(Q_{k}(s)\right)^{1/2}\, ds.
	\end{align*}  
	We obtain
	\begin{align*}
	v_{k}(t) \leq v_{k}(t_{0}) e^{- \int_{t_{0}}^{t} \left(Q_{k}(s)\right)^{1/2}\, ds}.
	\end{align*}
	In the region $s \geq E_{k}^{1/(2n+2)}=T_{k}$ we have $\left( s^{2n+1} -E_{k}^{1/2} s^{n}\right)^{2} \leq s^{4n+2} - E_{k}s^{2n}$.
	We set $Z_{0} = \delta^{-2} E_{k}^{1/(2n+2)} (\geq T_{k})$, $\delta \leq 1/2$. Taking $t_{0} = Z_{0}$ in the above formula we have
	\begin{align*}
	v_{k}(t) &\leq v_{k}(Z_{0}) e^{- \int_{Z_{0}}^{t} \left(Q_{k}(s)\right)^{1/2}\, ds}
	\leq v_{k}(Z_{0}) e^{- \int_{Z_{0}}^{t}  (s^{2n+1}
                   -E_{k}^{1/2} s^{n}) \, ds}\\
	&\leq v_{k}(Z_{0}) e^{\frac{Z_{0}^{2n+2}}{2n+2} - E_{k}^{1/2} \frac{Z_{0}^{n+1}}{n+1}} 
	e^{-\frac{t^{2n+2}}{2n+2} + E_{k}^{1/2} \frac{t^{n+1}}{n+1}} 
	= C(Z_{0}) e^{-\frac{t^{2n+2}}{2n+2} + E_{k}^{1/2} \frac{t^{n+1}}{n+1}} .
	\end{align*}
	Let $B = \frac{1- 2 \delta^{2(n+1)}}{2n+2}$, we have
	\begin{align*}
	v_{k}(t)  \leq C(Z_{0}) e^{- Bt^{2n+2}} , \quad \text{ for } t
          \geq Z_{0}.
	\end{align*}
	We remark that
	\begin{align*}
	C(Z_{0}) =  v_{k}(Z_{0}) e^{\frac{\delta^{-4(n+1)}}{2n+2} E_{k} - E_{k} \frac{\delta^{-2(n+1)}}{n+1}}
	\leq  \|v_{k}\|_{\infty} e^{\delta^{-4(n+1)} \frac{E_{k}}{2n+2}}.
	\end{align*}
	On the other side if $ t \in [0, Z_{0}]$ we have
	\begin{align*}
	v_{k}(t) \leq \|v_{k}\|_{\infty} e^{B\delta^{-4(n+1)} E_{k}}.
	\end{align*}
	Hence we have obtained the estimate
$$ 
|v_{k}(t) | \leq C_{0}^{k+1} e^{- B t^{2(n+1)}}. 
$$
where $ C_{0} $, $ B $ are suitable positive constants independent of
$ k $. 
\end{proof}
\begin{proposition}
  \label{est-v'}
	The following estimate holds
	\begin{align}
	|v'_{k}(t)| \leq C_{1} C_{0}^{k} e^{-B't^{2n+2}},
	\end{align}
	where $0 < B' < B$, $ B $, $ C_{0} $ are given in Proposition
        \ref{Exp-Est_Eig} and $C_{1} $ suitable independent of $ k $
        and such that $ C_{1} = \mathscr{O}((B-B')^{-2}) $. 
\end{proposition}
\begin{proof}
	Let $t> T_{k}$. By the Mean-Value Theorem we have
	\begin{align*}
	|v_{k}(t+1)-v_{k}(t)|= |v_{k}'(t_{1})|, \qquad t< t_{1} < t+1.
	\end{align*} 
	Since both  $|v_{k}'(t)|$ and $|v_{k}(t)|$ tend to zero as
        $t\rightarrow \infty$, we have 
	\begin{align*}
	|v'_{k}(t+1)| \leq |v_{k}'(t_{1})| \leq |v_{k}(t)| +|v_{k}(t+1)| \leq 2 | v_{k}(t)|.
	\end{align*}
	Arguing as above we have
	\begin{align*}
	|v_{k}'(t+1)-v_{k}'(t)|= |v_{k}''(t_{2})|= |t_{2}^{2n}(E_{k} - t_{2}^{2n+2}) v_{k}(t_{2})|, \qquad t< t_{2} < t+1.
	\end{align*} 
	For $t \geq T_{k}$ we obtain
	\begin{align*}
	|v'_{k}(t)|  &\leq |v'_{k}(t+1)| + |t_{2}^{2n}(E_{k} - t_{2}^{2n+2}) v_{k}(t_{2})|\\
	&\leq   | v_{k}(t)| \left( 2  + t^{4n+2}  (E_{k} t^{-(2n+2)} +
   1) \right) 
	\leq 2 |v_{k}(t)| \left( 1 + t^{4n+2}\right).
	\end{align*}
	We apply the result in Proposition \ref{Exp-Est_Eig}:
	\begin{align*}
|v'_{k}(t)| \leq 2 \left( 1 + t^{4n+2}\right) C_{0}^{k+1} e^{- B
          t^{2n+2}}. 
	\end{align*}
	Then we may conclude that
$$ 
|v'_{k}(t)| \leq C_{1} C_{0}^{k} e^{- B' t^{2n+2}} , 
$$
where $ 0 < B' < B $ and $ C_{1} = \mathscr{O}((B-B')^{-2}) $. This
proves the assertion.
\end{proof}
We consider $v_{k}(z) $, $z \in \mathbb{C}$, which is the entire
continuation of $ v_{k}(t) $. We are interested to the \textit{order}
of $v_{k}(z)$. We recall that an entire function $f(z)$ is of finite order if there exist  $C, \, \alpha > 0$ such that
\begin{align*}
|f(z)| \leq e^{|z|^{\alpha}} , \qquad \forall \, |z| \geq C.
\end{align*}
We have
\begin{proposition}
  \label{finite-Type}
	$v_{k}(z)$ has order $2n+2$, is of finite type, and the following estimate holds
	\begin{align}
          \label{vk-est-complx}
	|v_{k}(z)|\leq C^{k+1} e^{c_{1}|z|^{2n+2}},
	\end{align} 
	for some $c_{1}$ and $C $ positive, independent of $ k $. 
\end{proposition}
\begin{proof}
	We set $ f_{0}(z) = f_{0} = v_{k}(0)$ and
	\begin{align*}
	f_{\ell}(z) = f_{0}(z) + \int_{0}^{z} Q_{k}(\zeta) f_{\ell-1}(\zeta) (z-\zeta) d\zeta \qquad  \ell = 1,2,\dots\,\,.
	\end{align*}
	We have
	\begin{align*}
	&|f_{1}(z)\!-\!f_{0}(z)|\!=\! \left|  \int_{0}^{z}\!\!\! Q_{k}(\zeta) f_{0}(\zeta) (z-\zeta) d\zeta \right| \\
	&= \left| \int_{0}^{z}\!\!\! Q_{k}(\zeta) \left( v_{k}(0) + \zeta v'_{k}(0)\right) (z-\zeta) dz \right|
	\leq S \frac{\tilde{Q}_{k}(z) |z|^{2}}{2},
	\end{align*} 
	where $S = |v_{k}(0)| = | f_{0}| $ and $\tilde{Q}_{k}(z) = |z|^{2n}\left( |z|^{2n+2} +E_{k}\right)$.
	Using the above estimate at the step two we have
	\begin{align*}
	|f_{2}(z) - f_{1}(z)| = \left|\int_{0}^{z} Q_{k}(\zeta)  \left( f_{1}(\zeta)- f_{0}(\zeta)\right) (z-\zeta) d\zeta \right|
	\leq S \frac{\tilde{Q}^{2}(z) |z|^{4}}{4!}. 
	\end{align*}
	By induction argument we obtain 
	\begin{align*}
	|f_{\ell}(z) -f_{\ell -1}(z)| \leq S \frac{\tilde{Q}^{\ell}(z) |z|^{2\ell}}{(2\ell)!}, \qquad \forall \ell \geq 2.
	\end{align*}
	Due to the above estimate we have that the series $\sum_{\ell \geq 1} \left( f_{\ell}(z)- f_{\ell-1}(z)\right)$
	converges uniformly on compact sets and the term by term
        differentiation is permitted.
        Hence we may take
	a term by term second derivative of it. We consider the function 
	\begin{align*}
	f(z) = f_{0}(z) +\sum_{\ell= 1}^{\infty} \left(f_{\ell}(z) -f_{\ell-1}(z)\right).
	\end{align*}
	We have that
	\begin{align*}
	f''(z) = Q_{k}(z) f(z).
	\end{align*}
	$f(z)$ satisfies our equation with the same initial
        conditions at $z=0$ as $ v_{k}(z) $, so that
	\begin{align*}
	v_{k}(z) =  f_{0}(z) +\sum_{\ell= 1}^{\infty} \left(f_{\ell}(z) -f_{\ell-1}(z)\right).
	\end{align*} 
	If $|z| \geq c E_{k}^{1/(2n+2)}$ we have
	\begin{align*}
	|v_{k}(z)|& \leq S \left(1+ \sum_{\ell =1}^{\infty} \frac{\left(c_{1}|z|^{2n+2}\right)^{2\ell}}{(2\ell)!}\right)
	=S \cosh(c_{1}|z|^{2n+2}) \leq C e^{c_{1}|z|^{2n+2}} ,
	\end{align*}
        where $ c_{1} \geq 1 + c^{-2(n+1)} $ and $ C > 0 $ does not
        depend on $ k $.
        
	On the other side since $v_{k}(z)$ is a holomorphic function
        in the bounded domain $|z| \leq c E^{1/(2n+2)}_{k}$, its
        absolute value attains
	its maximum at some points on the boundary of the disc $|z| \leq c E_{k}^{1/(2n+2)}$:
	\begin{align*}
	|v_{k}(z)| \leq C e^{c_{2} E_{k}} \leq C_{2}^{k+1}.
	\end{align*} 
	We conclude that
	\begin{align*}
	|v_{k}(z)|\leq C_{2}^{k+1} e^{c_{1}|z|^{2n+2}},
	\end{align*}
	where $C_{2}$ is a suitable positive constant independent of $
        k $. This shows that the order of $v_{k}(z)$ is less or equal
        than $2n+2$; 
	on the other hand Proposition \ref{Exp-Est_Eig} shows that the
        order is greater or equal than $2n+2$.
\end{proof}
\noindent
Due to the Propositions \ref{Exp-Est_Eig} and \ref{finite-Type} and by the Remark at page 220 of \cite{Gelfand_Shilov} we have
\begin{theorem}\label{G-S_Char}
For any $ k $ the  function $v_{k}(t)$ belongs to the Gel'fand-Shilov  space
$S^{\frac{1}{2n+2}}_{\frac{2n+1}{2n+2}}(\mathbb{R})$ and moreover satisfies the
estimates
\begin{equation}
\label{Gelvk}
| t^{\alpha} \partial_{t}^{\beta} v_{k}(t)| \leq
C^{k+\alpha+\beta+1}_{v} \alpha!^{\frac{1}{2n+2}}
\beta!^{\frac{2n+1}{2n+2}} .
\end{equation}
\end{theorem}
\begin{proof}
The proof is contained in \cite{Gelfand_Shilov} and we sketch it here
for convenience.

By estimates \eqref{vk-est-real}, \eqref{vk-est-complx} we may apply
Theorem 1 of \cite{Gelfand_Shilov}, p. 213, and conclude that on a
domain of the complex plane of the form $ |s| \leq K_{1} (1 + |t|) $,
$ K_{1} $ a suitable positive constant independent of $ k $, we have
\begin{equation}
\label{vk-GS-1}
| v_{k}(t+is) | \leq C_{3}^{k+1} e^{-B' t^{2n+2}},
\end{equation}
where $ C_{3} = \max\{C, C_{0}\} $ and $ B' $ differs as little as
desired from $B $.

Applying Theorem 2 of \cite{Gelfand_Shilov}, p. 216, because of the
above estimate, we obtain that
\begin{equation}
\label{vk-GS-2}
| v_{k}(t+is) | \leq C_{3}^{k+1} e^{-B' t^{2n+2} + c_{2} s^{2n+2}},
\end{equation}
for any $ t+is = z \in \mathbb{C} $. Here $ c_{2} $ is a constant
depending on $ B' $, $ c_{1} $ and $ K_{1} $.

Finally we apply Theorem 3 of \cite{Gelfand_Shilov}, p. 219, and
obtain that
\begin{equation}
\label{vk-GS-3}
| \partial_{t}^{\beta} v_{k}(t) | \leq C_{3}^{k+1} B^{\beta}
\beta^{\beta \frac{2n+1}{2n+2}} e^{-B'' t^{2n+2}},
\end{equation}
where $ B'' $ differs as little as desired from $ B' $ and $ B $ is a
positive constant independent of $ k $.

This concludes the proof of the theorem.
\end{proof}
%

%
%
\subsection{Technical Lemmas}
\renewcommand{\theequation}{\thesubsection.\arabic{equation}}
\setcounter{equation}{0} \setcounter{theorem}{0}
The following paragraph is devoted to list several technical results used in the first two Sections.
\begin{lemma}\label{Tech_L1}
	Let $\theta$ be a positive real number.  Then for every positive integer $p$ the following identity holds
	\begin{align}\label{Form_L1}
	\left[\rho^{-\theta} \left(1- \theta  + \rho \partial_{\rho}\right)\right]^{p} = \rho^{-\theta p}\sum_{\ell=1}^{p}(-\theta)^{l-1} b_{p,\ell} 
	\left(1- \theta  + \rho \partial_{\rho}\right)^{p+1-\ell},
	\end{align}
	where the constants $b_{p,\ell }$ satisfy the estimate 
	$$
	b_{p,\ell}\leq 2^{p-1}(p-1)^{\ell}.
	$$
	In particular $b_{p,1}= 1$, $b_{p,2}= p(p-1)/2$, $b_{p,p} = ( p-1)!$.
\end{lemma}
\begin{proof}
By induction we have that the coefficients $b_{p,\ell}$ satisfy the recurrence relations: $b_{p,1}= 1$,  $b_{p,p} = b_{p-1,p-1}( p-1)$ and 
$ b_{p,\ell} = b_{p-1,\ell}+ (p-1)b_{p-1,\ell-1}$, $\ell = 3,\dots, p-1 $, where $b_{p-1,\ell}$ are the coefficients of $(AB)^{p-1}$.
By induction we have the estimate.
\end{proof}
\begin{lemma}\label{Tech_L2}
	 Let $\nu$ be a positive integer number. Then the following identity holds
	 \begin{align*}
	 \left(\rho\partial_{\rho}\right)^{\nu} = \sum_{i=1}^{\nu}d_{\nu,i}\, \rho^{\nu+1-i}\partial_{\rho}^{\nu+1-i},
	 \end{align*}
	 where the constants $d_{\nu,i}$ satisfy the estimate 
	 $$
	 d_{\nu,i}\leq 2^{\nu+i}(\nu+1-i)^{i}.
	 $$
	 In particular $d_{\nu,1}= 1 = d_{\nu,\nu}$, $d_{\nu,2}=\nu(\nu-1)/2$ and  $d_{\nu,i}$,  $i=3,\dots,\nu-1$. 
\end{lemma}
\begin{proof}
	By induction we have that the coefficients $d_{\nu,i}$ satisfy the recurrence relations: $d_{\nu,1}= 1 = d_{\nu,\nu}$, $d_{\nu,2}=\nu(\nu-1)/2$
	and  $d_{\nu,i}= d_{\nu-1,i}+(\nu+1-i)d_{\nu-1,i-1}$, $i=3,\dots,\nu-1$, where $d_{\nu-1, i}$ are the
	coefficients of $(\rho\partial_{\rho})^{\nu-1}$. By induction we have the estimate.
\end{proof}
\begin{lemma}\label{Tech_L3}
	Let $\theta$, $f$, $q$ and $\gamma$ real number, $\theta>0$. Then for every integer $p$ the following identity holds
	\begin{align}\label{F-2}
	& \left[\rho^{-\theta} \left(1- \theta  + \rho \partial_{\rho}\right)\right]^{p} \rho^{q} \!\!\left[t^{f}u(t,\rho)\right]_{|_{t=\rho^{\gamma}x}}\!\!
	=\rho^{q-\theta p}\!\! \left[t^{f}\!\sum_{i=0}^{p}\rho^{-i}\mathscr{P}_{i}(t\partial_{t})\partial_{\rho}^{p-i} u(t,\rho)\right]_{|_{t=\rho^{\gamma}x}}
	\\ \nonumber
	&\qquad\qquad= \rho^{q-\theta p}  \left[t^{f} \left( \partial_{\rho}^{p} +\frac{p}{\rho} \left( \frac{p+1}{2}(1-\theta) +q +\gamma f + \gamma t\partial_{t}\right) \partial_{\rho}^{p-1}
	\right.\right.
	\\ \nonumber
	&\left.\left.\qquad\qquad\qquad\qquad\qquad\qquad\qquad\qquad\qquad\qquad
	+\sum_{i=2}^{p}\rho^{-i}\mathscr{P}_{i}(t\partial_{t})\partial_{\rho}^{p-i} \right) u(t,\rho)\right]_{|_{t=\rho^{\gamma}x}},
	\end{align}
	where
	\begin{align}\label{OP-Pi}
	\mathscr{P}_{i}(t\partial_{t})
	=\sum_{j=0}^{i} \text{\textcursive{ p}}_{i,j}(t\partial_{t})^{j}. 
	\end{align}
	The coefficients $\text{\textcursive{ p}}_{i,j}$, $ i=2,\dots,p-1$, are given by
	\begin{align}\label{Coef_1}
	\gamma^{j}\sum_{\nu=j}^{i}\sum_{\mu=\nu}^{i} \!\frac{(p+\nu\!-\!\mu)!}{j!(\nu-j)!(p-\mu)!}(-\theta)^{\mu-\nu}(1\!-\!\theta\! +\!q+\!\gamma f)^{\nu-j}
	b_{p,\mu-\nu+1}d_{p-\mu,i-\mu+1},
	\end{align}
	and
	\begin{align}\label{Coef_2}
	\text{\textcursive{ p}}_{p,j} = \sum_{\nu=j}^{p}\binom{\nu}{j}(-\theta)^{\nu-j}(1-\theta +q+\gamma f)^{\nu-j}b_{p,p-\nu+1}.
	\end{align}
	The constants $b_{p,\mu-\nu+1}$ and $d_{p-\mu,i-\mu+1}$ are the same of the previous Lemmas.  
\end{lemma}
\begin{proof}
	We have
	\begin{align*}
	&\left(1-\theta + \rho\partial_{\rho} \right) \rho^{q}\left[t^{f}u(t,\rho) \right]_{|_{t=\rho^{\gamma}x}}
	=\rho^{q}\left(1-\theta +q + \rho\partial_{\rho} \right )\left[t^{f}u(t,\rho) \right]_{|_{t=\rho^{\gamma}x}}
	\\
	&=\rho^{q} \left[t^{f}\left( 1-\theta+ q +\gamma f+ \gamma t\partial_{t} +\rho\partial_{\rho}\right) u(t,\rho)\right]_{|_{t=\rho^{\gamma}x}}.
	\end{align*}
	Applying the Lemma \ref{Tech_L1} we can rewrite the left hand side of (\ref{F-2}) in the following form
	\begin{align}\label{F-1}
	&\rho^{-\theta p}\sum_{\ell=1}^{p}(-\theta)^{l-1} b_{p,\ell} 	\left(1- \theta  + \rho \partial_{\rho}\right)^{p+1-\ell} \rho^{q}\left[t^{f}u(t,\rho) \right]_{|_{t=\rho^{\gamma}x}}\\
	\nonumber
	&=\rho^{q-\theta p} \left[t^{f}\sum_{\ell=1}^{p}(-\theta)^{l-1} b_{p,\ell} \left( 1-\theta+ q +\gamma f+ \gamma t\partial_{t} +\rho\partial_{\rho}\right)^{p+1-\ell}u(t,\rho)\right]_{|_{t=\rho^{\gamma}x}}.
	\end{align}
	We observe that if $Q_{1}$ and $Q_{2}$ be two operators such that $\left[Q_{1},Q_{2}\right]=0$ for every positive integer $\nu$ we have
	\begin{align*}
	(Q_{1}+Q_{2})^{\nu}= \sum_{i=0}^{\nu}\binom{\nu}{i} Q_{1}^{i}Q_{2}^{\nu-i}. 
	\end{align*}
	Taking advantage from the above formula and from the Lemma \ref{Tech_L2} we obtain (\ref{F-2}) from (\ref{F-1}). 
\end{proof}
\begin{lemma}\label{L_Fund_Sol}
	Let $\Theta_{k}$ be the ordinary differential equation of order $2m$, $m\in \mathbb{Z}_{+}$, given by
	\begin{align}\label{ODE}
\Theta_{k} = \left(\frac{d}{d \rho}\right)^{2m} + \left(\frac{2m i}{2m-1}\right)^{2m}E_{k},
	\end{align}
	where $E_{k}=4k(n+1)+2n+1$ is the eigenvalue of the
        eigenfunction $v_{k}(t)$, (see (\ref{Ev_Eigf}).)
	We shall use the following fundamental solution of $\Theta_{k}$
	\begin{align}\label{Sol_Fun}
	G_{k}(\rho) = (i c_{k})^{-(2m-1)} \frac{1}{m}
\sum_{j=0}^{[\frac{m-1}{2}]} 2^{-[\sin \theta_{j}]}  e^{- c_{k} \sin \theta_{j} |\rho|}
          \sin\left(c_{k} |\rho| \cos \theta_{j} + \theta_{j}\right) ,
	\end{align}
	where $c_{k}= 2m (2m-1)^{-1}E_{k}^{1/2m}$, $ \theta_{j} =
        \frac{\pi}{2m}(1+2j) $, $j = 0, 1, \ldots, 2m-1$ and $ [x] =
        \max \{ n \in \mathbb{Z}_{+} \ | \ n \leq x\} $, $ x \geq 0 $.
We point out that when $ m $ is odd the summand in \eqref{Sol_Fun}
corresponding to $ j = \frac{m-1}{2} $ is $ e^{- c_{k} |\rho|} $ and $
[\sin \theta_{j}] = 0 $ if $ j < \frac{m-1}{2} $, $ [\sin
\theta_{j}] = 1 $ when $ j = \frac{m-1}{2} $. 

	Moreover we have
	\begin{align}
          \label{FunEst}
	|G_{k}(\rho)|\leq  
        c_{k}^{-(2m-1)} e^{-c_{k} \sin(\pi/2m) |\rho|} .
	\end{align}
Moreover for every $\ell$, $ 0 \leq \ell \leq 2m-1 $, we have
	\begin{equation}\label{Der-Fun-odd}
	G_{k}^{(\ell)}(\rho) = i^{-(2m-1)} c_{k}^{-(2m - 1 - \ell)} \frac{1}{m}
        \sum_{j=0}^{[\frac{m-1}{2}]} 2^{-[\sin \theta_{j}]} e^{- c_{k} \sin \theta_{j}
          |\rho|} \cos\left(c_{k} |\rho| \cos \theta_{j} +
          (\ell+1) \theta_{j}\right) \sign(\rho) ,
	\end{equation}
        if $ \ell $ is odd and
	\begin{equation}\label{Der-Fun-even}
	G_{k}^{(\ell)}(\rho) = i^{-(2m-1)} c_{k}^{-(2m - 1 - \ell)} \frac{1}{m}
        \sum_{j=0}^{[\frac{m-1}{2}]} e^{- c_{k} \sin \theta_{j}
          |\rho|} \sin\left(c_{k} |\rho| \cos \theta_{j} +
          (\ell+1) \theta_{j}\right) ,
	\end{equation}
        if $ \ell $ is even.
\end{lemma}
\begin{proof}
Using e.g. the residue theorem we have that
$$ 
G_{k}(\rho) = \frac{1}{2\pi} \int_{\mathbb{R}} e^{i \rho \sigma}
\frac{(-1)^{m}}{\sigma^{2m} + c_{k}^{2m}} d\sigma.
$$
If $ \rho \geq 0 $ we have
$$ 
G_{k}(\rho) = \left(i c_{k}\right)^{-(2m-1)} i \sum_{j=0}^{m-1}
\frac{1}{2m \ e^{i(2m-1) \theta_{j}}} e^{ic_{k} e^{i \theta_{j}}  \rho} .
$$
If $ \rho \leq 0 $ we have
$$ 
G_{k}(\rho) = - \left(i c_{k}\right)^{-(2m-1)}  i \sum_{j=m}^{2m-1}
\frac{1}{2m \ e^{i(2m-1) \theta_{j}}} e^{ic_{k} e^{i \theta_{j}}  \rho} .
$$
Since $ \theta_{\ell+m} = \theta_{\ell} + \pi $ for $ 0 \leq \ell \leq
m-1  $, we deduce that
\begin{equation}
  \label{Gk}
G_{k}(\rho) =  \left(i c_{k}\right)^{-(2m-1)}  i \sum_{j=0}^{m-1}
\frac{1}{2m \ e^{i(2m-1) \theta_{j}}} e^{ic_{k} e^{i \theta_{j}}  |\rho|} .
\end{equation}
Further we have that $e^{i(2m-1) \theta_{j}} = - e^{-i\theta_{j}}  $
for every $ j $, $ 0 \leq j \leq 2m-1 $, and $ \theta_{m-1-\ell} = \pi
- \theta_{\ell}$. Plugging this into the expression of $ G_{k} $ we,
if $ m $ is even, can sum the $ j $-th summand and the $ (m-1-j)$-th
summand obtaining that
\begin{multline*}
- e^{i \theta_{j}} e^{i c_{k} e^{i \theta_{j}} |\rho|} - e^{i
  \theta_{m-1-j}} e^{i c_{k} e^{i \theta_{m-1-j}} |\rho|}
\\
=
- e^{i \theta_{j}} e^{i c_{k} e^{i \theta_{j} }|\rho|}
+ e^{-i \theta_{j}} e^{-i c_{k} e^{- i \theta_{j}}|\rho|}
\\
=
- e^{i \theta_{j} + i c_{k} \cos \theta_{j} |\rho| } e^{- c_{k} \sin \theta_{j} |\rho|}
+ e^{-i \theta_{j} - i c_{k} \cos \theta_{j}} e^{- c_{k} \sin
  \theta_{j}|\rho|}
\\
=
-2i e^{- c_{k} \sin \theta_{j}|\rho|} \left(\sin \theta_{j} \cos \left(c_{k}
  \cos \theta_{j} |\rho| \right) + \cos \theta_{j} \sin\left(c_{k}
  \cos \theta_{j} |\rho| \right) \right)
\\
=
-2i e^{- c_{k} \sin \theta_{j}|\rho|} \sin\left( c_{k}
  \cos \theta_{j} |\rho| + \theta_{j}\right) .
\end{multline*}
Plugging this into \eqref{Gk} and taking into account that the number
of summands is even, we obtain
$$ 
G_{k}(\rho) =  \left(i c_{k}\right)^{-(2m-1)} \frac{1}{m}
\sum_{j=0}^{\frac{m}{2}-1}  e^{- c_{k} \sin \theta_{j}|\rho|} \sin\left( c_{k}
  \cos \theta_{j} |\rho| + \theta_{j}\right) ,
$$
which is \eqref{Sol_Fun} for $ m $ even.

Consider now the case of $ m $ odd. The number of summands in
\eqref{Gk} is odd, so that we may proceed as above for all but one
summand, obtained when $ j = \frac{m-1}{2} $. Now for this value of $
j $, $ \theta_{j} = \frac{\pi}{2} $, which implies that the
corresponding summand is written as
$$ 
\left(i c_{k}\right)^{-(2m-1)} \frac{1}{2m}  e^{- c_{k} |\rho|} 
$$
which gives \eqref{Sol_Fun}. Deriving \eqref{FunEst} is
straightforward. 

Let us now turn to the derivatives of $ G_{k} $ when $ m $ is odd. We
have
\begin{multline*}
\frac{d}{d \rho} G_{k}(\rho) = (i c_{k})^{-(2m-1)} \frac{1}{m}
\sum_{j=0}^{[\frac{m-1}{2}]} 2^{-[\sin \theta_{j}]}  e^{- c_{k} \sin \theta_{j} |\rho|}
\left[ - c_{k} \sin \theta_{j}
  \sin\left(c_{k} |\rho| \cos \theta_{j} + \theta_{j}\right)
\right.
\\
\left.
  +
c_{k} \cos \theta_{j} \cos \left(c_{k} |\rho| \cos \theta_{j} +
  \theta_{j}\right) \right] \sign(\rho)
\\
=
i^{-(2m-1)} c_{k}^{-(2m-2)} \frac{1}{m}
\sum_{j=0}^{[\frac{m-1}{2}]} 2^{-[\sin \theta_{j}]} e^{- c_{k} \sin \theta_{j} |\rho|}
  \cos\left(c_{k} |\rho| \cos \theta_{j} + 2 \theta_{j}\right)
  \sign(\rho) ,
\end{multline*}
which is \eqref{Der-Fun-odd} when $ \ell = 1 $. Again
\begin{multline*}
\frac{d^{2}}{d \rho^{2}} G_{k}(\rho) = i^{-(2m-1)}  c_{k}^{-(2m-2)} \frac{1}{m}
\sum_{j=0}^{[\frac{m-1}{2}]} 2^{-[\sin \theta_{j}]} e^{- c_{k} \sin \theta_{j} |\rho|}
\left[ - c_{k} \sin \theta_{j}
  \cos \left(c_{k} |\rho| \cos \theta_{j} + 2 \theta_{j}\right)
\right.
\\
\left.
  -
c_{k} \cos \theta_{j} \sin \left(c_{k} |\rho| \cos \theta_{j} +
  2\theta_{j}\right) \right]
\\
+
i^{-(2m-1)} c_{k}^{-(2m-2)} \frac{1}{m}
\sum_{j=0}^{[\frac{m-1}{2}]} 2^{-[\sin \theta_{j}]} e^{- c_{k} \sin \theta_{j} |\rho|}
  \cos\left(c_{k} |\rho| \cos \theta_{j} + 2 \theta_{j}\right)
  2 \delta(\rho)
\\
= i^{-(2m-3)}  c_{k}^{-(2m-3)} \frac{1}{m}
\sum_{j=0}^{[\frac{m-1}{2}]} 2^{-[\sin \theta_{j}]} e^{- c_{k} \sin \theta_{j} |\rho|}
  \sin \left(c_{k} |\rho| \cos \theta_{j} + 3 \theta_{j}\right)
\\
+
i^{-(2m-1)} c_{k}^{-(2m-2)} \frac{2}{m}
\sum_{j=0}^{[\frac{m-1}{2}]} 2^{-[\sin \theta_{j}]}
  \cos\left(2 \theta_{j}\right) \delta(\rho) .
\end{multline*}
The last sum in the above formula gives zero:
$$ 
\sum_{j=0}^{[\frac{m-1}{2}]} 2^{-[\sin \theta_{j}]}
  \cos\left(2 \theta_{j}\right) = 0.
$$
In fact we have---we recall that $ m $ is odd and that $
\theta_{\frac{m-1}{2}} = \frac{\pi}{2m}(1 + 2 \frac{m-1}{2}) = \frac{\pi}{2} $---
\begin{multline*}
\sum_{j=0}^{\frac{m-1}{2}} 2^{-[\sin \theta_{j}]}
  \cos\left(2 \theta_{j}\right) = \sum_{j=0}^{\frac{m-3}{2}} 
  \cos\left(2 \theta_{j}\right)  + \frac{1}{2} \cos(2
  \theta_{\frac{m-1}{2}})
\\
= - \frac{1}{2} + \Re \left( e^{i \frac{\pi}{m}} \sum_{j=0}^{\frac{m-3}{2}} 
e^{i \frac{2\pi}{m} j}  \right)
=  - \frac{1}{2} + \Re \left( e^{i \frac{\pi}{m}} \frac{1 - e^{i \pi
      \frac{m-1}{m}}}{1- e^{i \frac{2\pi}{m}}}  \right)
\\
=
 - \frac{1}{2} + \Re \left( \frac{1 +  e^{ - i
      \frac{\pi}{m}}}{e^{- i \frac{\pi}{m}}- e^{i \frac{\pi}{m}}}
\right)
=
- \frac{1}{2} + \Re \frac{1 + \cos \frac{\pi}{m} - i \sin
  \frac{\pi}{m}}{ -2i \sin \frac{\pi}{m}} = 0.
\end{multline*}
This proves \eqref{Der-Fun-even} for $ \ell = 2 $. For $ \ell $ odd we
see that the $ \ell $-th derivative is a multiple of $ \sign (\rho) $
and \eqref{Der-Fun-odd} can be derived with an argument completely
analogous to the above. Assume that $ \ell $ is even. Then the $ \ell
$-th derivative is the sum of a term of the form \eqref{Der-Fun-even}
and a multiple of the distribution $ \delta $. The latter contains the
sum
$$ 
\sum_{j=0}^{\frac{m-1}{2}} 2^{-[\sin \theta_{j}]}
  \cos\left(\ell \theta_{j}\right) .
$$
This is zero since arguing as above we isolate the summand
corresponding to $ j = \frac{m-1}{2} $ obtaining the value $
\frac{(-1)^{\frac{\ell}{2}}}{2} $. The remaining sum then gives the
opposite using the same argument as above.

This proves \eqref{Der-Fun-even} and \eqref{Der-Fun-odd} when $ m $ is
odd. 

Assume $ m $ even. Then the factor $ 2^{-[\sin \theta_{j}]} \equiv 1
$. Again, arguing as above, we find that the distribution $ \delta $
has a factor containing the sum
$$ 
\sum_{j=0}^{\frac{m}{2} - 1}  \cos\left(\ell \theta_{j}\right) = 0 .
$$
This completes the proof of the lemma.

\end{proof}
\begin{lemma}\label{L-j_Est}
	Let $a$ and $b$ be two positive constants such that $a >b$. 
	Then there is a positive constant $C_{0}$ such that for every $R \geq C_{0}(j+1)$, $j\in \mathbb{Z}_{+}$, the following inequality holds
	\begin{align}\label{Int_Inq-j}
	\int_{R}^{+\infty}  \frac{1}{\tau^{j}} e^{-a|\rho-\tau|-b\tau}d\tau \leq \left(\frac{1}{a}+\frac{2}{a-b}\right) \frac{e^{-b\rho}}{\rho^{j}}.
	\end{align}  
\end{lemma}
\begin{proof}
	We write the right hand side of (\ref{Int_Inq-j}) as
	\begin{align*}
	\int_{R}^{\rho} e^{a\tau-a\rho} \frac{1}{\tau^{j}} e^{-b\tau}d\tau + \int_{\rho}^{+\infty} e^{-a\tau+a\rho} \frac{1}{\tau^{j}} e^{-b\tau}d\tau
	\doteq  (I) + (II).
	\end{align*}
	Since $\rho < \tau$ then $\rho^{-1}>\tau^{-1}$ and $ e^{-b\tau}< e^{-b\rho}$ we have
	\begin{align*}
	(II) \leq \frac{e^{-b\rho}}{\rho^{j}}\int_{\rho}^{+\infty}e^{-a(\tau-\rho)}d\tau = \frac{e^{-b\rho}}{a\rho^{j}} \int_{0}^{+\infty} e^{-s}ds =  \frac{e^{-b\rho}}{a\rho^{j}}.
	\end{align*}
	Let us estimate of term $(I)$.
        We take $\tau^{-j}= e^{-j\log\tau}$ and we consider the function 
	$f(\tau)= 2^{-1}( a\tau-b\tau) -j\log\tau$. Since
        $f^{(1)}(\tau) =2^{-1}(a-b) -j\tau^{-1}$,  $f(\tau)$ is decreasing function for $\tau < 2(a-b)^{-1}j$
	and increasing function for $\tau > 2(a-b)^{-1}j$.
	Since $e^{\inf f(\tau)}\leq  e^{f(\tau)} \leq e^{\sup f(\tau)}$ taking $R> 2(a-b)^{-1}j $ so that $f(\tau)$ is increasing function in the region $[R,\rho] $
	we have that $\sup f(\tau) =f(\rho)$. We have 
	\begin{align*}
	e^{f(\tau)} \leq \frac{1}{\rho^{j}} e^{\frac{(a-b)}{2}\rho}.
	\end{align*}
	For $R > 2(a-b)^{-1}j $ we have
	\begin{align*}
	(I)\leq &\frac{1}{\rho^{j}} e^{-\frac{(a+b)}{2}\rho} \int_{R}^{\rho} e^{\frac{(a-b)}{2}\tau}d\tau
	\leq  \frac{1}{\rho^{j}} e^{-\frac{(a+b)}{2}\rho} \int_{0}^{\rho} e^{\frac{(a-b)}{2}\tau}d\tau\\
	&
	= \frac{1}{\rho^{j}} e^{-\frac{(a+b)}{2}\rho}  \frac{2}{a-b}\left(e^{\frac{(a-b)}{2}\rho}-1\right)
	\leq  \frac{2}{a-b}  \frac{1}{\rho^{j}} e^{-b\rho}.
	\end{align*}
	Summing up we obtain (\ref{Int_Inq-j}).
\end{proof}
\begin{remark}\label{Rk_L-j_Est}
	Setting $a = c_{k} \sin(\pi/2m)= 2m (2m-1)^{-1}E_{k}^{1/2m} \sin(\pi/2m)$ and $b=c_{0} \sin(\pi/2m)= 2m(2m-1)^{-1}E_{0}^{1/2m}(\sin \pi/2m)$,
	$E_{k}=  4k(n+1)+2n+1$, $k \in \mathbb{Z}_{+}$, in the above Lemma we have
	\begin{align}\label{Int_Inq-E}
	\int_{R}^{+\infty}\!\!\!  \frac{1}{\tau^{j}} e^{- c_{k} \sin(\pi/2m)|\rho-\tau|-c_{0} \sin(\pi/2m)\tau}d\tau 
	\leq \frac{C}{ (k+1)^{1/2m}\rho^{j}} e^{-c_{0}\sin(\pi/2m)\rho},
	\end{align}  
	for $R \geq C_{0}(j+1)$. Here $C$ denotes a suitable constant
        depending only on $m$;  $C_{0}$ can be chosen greater
        than two and independent of $j$ and $k$.
\end{remark}
\begin{lemma}\label{L-0_Est}
	Let $ G_{0}(\rho)$ be the fundamental  solution of $\Theta_{0}$ used in Lemma \ref{L_Fund_Sol}.
	Then there is a positive constant $C_{0}$ greater than two such that for every $R \geq C_{0}(j+1)$, $j \in \mathbb{Z}_{+}$, we have
	\begin{align}\label{Int_Inq-2}
	\left|\int_{R}^{+\infty} \!\!\!G_{0}(\rho-\tau) \frac{1}{\tau^{j+1}} e^{-c_{0} \lambda_{0}''\tau} d\tau 
	- h_{0}(\rho)- h_{m-1}(\rho)\right| 
	\leq C_{1} \frac{e^{-c_{0}\lambda_{0}''\rho}}{\rho^{j}(j+1)},
	\end{align}
	where $C_{1}$ is a positive constant independent by $j$, $c_{0}=  2m (2m-1)^{-1}E_{0}^{1/2m} $ and
	\begin{align*}
	&h_{0}(\rho)=(i c_{0})^{-(2m-1)} \frac{1}{2im} e^{i (c_{0}  \lambda_{0} \rho + \theta_{0})} \int_{R}^{+\infty}  e^{ic_{0} \lambda_{0}'\tau} \frac{1}{\tau^{j+1}} d\tau,
	\\
	&h_{m-1}(\rho)= (i c_{0})^{-(2m-1)} \frac{1}{2im}  e^{i (c_{0}  \lambda_{m-1} \rho + \theta_{m-1})}  \int_{R}^{+\infty}  e^{ic_{0} \lambda_{m-1}'\tau} \frac{1}{\tau^{j+1}} d\tau,
	\end{align*}
	with $\lambda_{0}=\lambda_{0}'+i \lambda_{0}''= \cos \theta_{0} + i \sin \theta_{0}$, $ \theta_{0} = \frac{\pi}{2m}$,
	$\lambda_{m-1} = \lambda_{m-1}'+ i\lambda_{m-1}'' = \cos \theta_{m-1}  + i \sin \theta_{m-1} $, $ \theta_{m-1} = \frac{\pi(2m-1)}{2m}$.
	%
\end{lemma}
\begin{proof}
	We recall that 
	\begin{align*}
	G_{0}(\rho) = (i c_{0})^{-(2m-1)} \frac{1}{m}
	\sum_{k=0}^{[\frac{m-1}{2}]} 2^{-[\sin \theta_{k}]}  e^{- c_{0} \sin \theta_{k} |\rho|}
	\sin\left(c_{0} |\rho| \cos \theta_{k} + \theta_{k}\right).
	\end{align*}
	We have
	\begin{multline}\label{estG0}
	\int_{R}^{+\infty} \!\!\!G_{0}(\rho-\tau) \frac{1}{\tau^{j+1}} e^{-c_{0} \lambda_{0}''\tau} d\tau
	= (i c_{0})^{-(2m-1)} \frac{1}{m}
	\\
	\times
	\sum_{k=0}^{[\frac{m-1}{2}]} 2^{-[\sin \theta_{k}]}  
	\int_{R}^{+\infty} \! \!\!\! 
	e^{- c_{0} \sin \theta_{k} |\rho-\tau|}
	\sin\left(c_{0} |\rho-\tau| \cos \theta_{k} + \theta_{k}\right)
	\frac{1}{\tau^{j+1}} e^{-c_{0} \lambda_{0}''\tau} d\tau.
	\end{multline}
	We begin to handle the term $k=0$. Since
	\begin{multline*}
	e^{- c_{0} \sin \theta_{0}|\rho|} \sin\left( c_{0}
	\cos \theta_{0} |\rho| + \theta_{0}\right)
	\\ 
	=
	e^{- c_{0} \sin \theta_{0}|\rho|} \left(\sin \theta_{0} \cos \left(c_{0}
	\cos \theta_{0} |\rho| \right) + \cos \theta_{0} \sin\left(c_{0}
	\cos \theta_{0} |\rho| \right) \right)
	\\
	=
	\frac{1}{2i} \left( e^{i \theta_{0} + i c_{0} \cos \theta_{0} |\rho| } e^{- c_{0} \sin \theta_{0} |\rho|}
	- e^{-i \theta_{0} - i c_{0} \cos \theta_{0}} e^{- c_{0} \sin \theta_{0}|\rho|}\right)
	\\
	\frac{1}{2i} \left( 
	e^{i \theta_{0}} e^{i c_{0} e^{i \theta_{0} }|\rho|}
	- e^{-i \theta_{0}} e^{-i c_{0} e^{- i \theta_{0}}|\rho|}
	\right)
	\\
	=
	\frac{1}{2i} \left( e^{i \theta_{0}} e^{i c_{0} e^{i \theta_{0}} |\rho|}
	 + e^{i \theta_{m-1}} e^{i c_{0} e^{i \theta_{m-1}} |\rho|}
	 \right)
	\\
	\\
	=
	\frac{1}{2i} \left( e^{i \theta_{0}} e^{i c_{0} \lambda_{0} |\rho|}
	+ e^{i \theta_{m-1}} e^{i c_{0} \lambda_{m-1} |\rho|},
	\right).
	\end{multline*}
	we have
	\begin{multline*}
	\int_{R}^{+\infty} \!  
	e^{- c_{0} \sin \theta_{0} |\rho-\tau|}
	\sin\left(c_{0} |\rho-\tau| \cos \theta_{0} + \theta_{0}\right)
	\frac{1}{\tau^{j+1}} e^{-c_{0} \lambda_{0}''\tau} d\tau
	\\
	=
	\frac{1}{2i} \left( 
	e^{i \theta_{0}} \int_{R}^{+\infty} \! e^{i c_{0} \lambda_{0} |\rho-\tau|} \frac{1}{\tau^{j+1}} e^{-c_{0} \lambda_{0}''\tau} d\tau
	+
	e^{i \theta_{m-1}}\int_{R}^{+\infty} \! e^{i c_{0} \lambda_{m-1} |\rho-\tau|} \frac{1}{\tau^{j+1}} e^{-c_{0} \lambda_{0}''\tau} d\tau
	\right)
	\\
	=
	\frac{1}{2i} \left[
	e^{i \theta_{0}} \left(
	\int_{R}^{\rho} \! e^{i c_{0} \lambda_{0} (\rho-\tau)} \frac{1}{\tau^{j+1}} e^{-c_{0} \lambda_{0}''\tau} d\tau
	+
	\int_{\rho}^{+\infty}  e^{i c_{0} \lambda_{0} (\tau-\rho)} \frac{1}{\tau^{j+1}} e^{-c_{0} \lambda_{0}''\tau} d\tau
	\right)
	\right.
	\\
	\left.
	+
	e^{i \theta_{m-1}}
	\left(	\int_{R}^{\rho} \! e^{i c_{0} \lambda_{m-1} (\rho-\tau)} \frac{1}{\tau^{j+1}} e^{-c_{0} \lambda_{0}''\tau} d\tau
	+
	\int_{\rho}^{+\infty} \! e^{i c_{0} \lambda_{m-1} (\tau-\rho)} \frac{1}{\tau^{j+1}} e^{-c_{0} \lambda_{0}''\tau} d\tau
	\right)
	\right]
	\\
	=
	\frac{1}{2i} \left[
	e^{i (c_{0}  \lambda_{0} \rho + \theta_{0})} \left(
	\int_{R}^{+\infty} \!  \frac{e^{-ic_{0} \lambda_{0}'\tau} }{\tau^{j+1}} d\tau
	+
	\int_{\rho}^{+\infty} \!  \frac{e^{-i c_{0} \lambda_{0}' \tau}}{\tau^{j+1}} d\tau
	+
	\int_{\rho}^{+\infty}  \frac{e^{i c_{0} \lambda_{0}' \tau}}{\tau^{j+1}} e^{-2c_{0} \lambda_{0}''\tau} d\tau
	\right)
	\right.
	\\
	\left.
	+
	e^{i (c_{0}  \lambda_{m-1} \rho + \theta_{m-1})} \left(
	\int_{R}^{+\infty} \!  \frac{e^{-ic_{0} \lambda_{m-1}'\tau} }{\tau^{j+1}} d\tau
	+
	\int_{\rho}^{+\infty} \!  \frac{e^{-i c_{0} \lambda_{m-1}' \tau}}{\tau^{j+1}} d\tau
	\right. \right.
	\\
	\left.	\left.
	+
	\int_{\rho}^{+\infty}  \frac{e^{i c_{0} \lambda_{m-1}' \tau}}{\tau^{j+1}} e^{-2c_{0} \lambda_{0}''\tau} d\tau
	\right)
	\right],
	\end{multline*}
	where we used that $\lambda_{m-1}'= -\lambda_{0}'$ and $ \lambda_{m-1}''= \lambda_{0}''$.
	We obtain that
	\begin{multline*}
	(i c_{0})^{-(2m-1)} \frac{1}{m}
	\int_{R}^{+\infty} \!  
	e^{- c_{0} \sin \theta_{0} |\rho-\tau|}
	\sin\left(c_{0} |\rho-\tau| \cos \theta_{0} + \theta_{0}\right)
	\frac{1}{\tau^{j+1}} e^{-c_{0} \lambda_{0}''\tau} d\tau
	\\
	\hspace*{-26em}=h_{0}(\rho) +  h_{m-1}(\rho) 
	\\
	+(i c_{0})^{-(2m-1)}
	\frac{1}{2im} \left[
	e^{i (c_{0}  \lambda_{0} \rho + \theta_{0})} \left(
	\int_{\rho}^{+\infty} \!  \frac{e^{-i c_{0} \lambda_{0}' \tau}}{\tau^{j+1}} d\tau
	+
	\int_{\rho}^{+\infty}  \frac{e^{i c_{0} \lambda_{0}' \tau}}{\tau^{j+1}} e^{-2c_{0} \lambda_{0}''\tau} d\tau
	\right)
	\right.
	\\
	\left.
	+
	e^{i (c_{0}  \lambda_{m-1} \rho + \theta_{m-1})} \left(
	\int_{\rho}^{+\infty} \!  \frac{e^{-i c_{0} \lambda_{m-1}' \tau}}{\tau^{j+1}} d\tau
	+
	\int_{\rho}^{+\infty}  \frac{e^{i c_{0} \lambda_{m-1}' \tau}}{\tau^{j+1}} e^{-2c_{0} \lambda_{0}''\tau} d\tau
	\right)
	\right]
	\\
	=h_{0}(\rho) +  h_{m-1}(\rho) +(I)
	\end{multline*}
	The absolute value of $(I)$ can be estimated by
	\begin{align*}
	C'\frac{e^{-c_{0}\lambda_{0}''\rho}}{\rho^{j}j},
	\end{align*}
	where $C'$ is a positive constant independent by $j$.\\
	Concerning the other terms in sum (\ref{estG0}), 
	$$
	\int_{R}^{+\infty} \! \!\!\! 
	e^{- c_{0} \sin \theta_{k} |\rho-\tau|}
	\sin\left(c_{0} |\rho-\tau| \cos \theta_{k} + \theta_{k}\right)
	\frac{1}{\tau^{j+1}} e^{-c_{0} \lambda_{0}''\tau} d\tau, 
	\quad k =2, \dots, [\frac{m-1}{2}],
	$$
	we can apply the Lemma \ref{L-j_Est}. Since $ \sin \theta_{k} > \sin \theta_{0}= \lambda_{0}''$ the absolute value of the above quantity
	can be bonded by
	\begin{align*}
	\int_{R}^{+\infty} \!\!\! e^{-c_{0} \sin \theta_{k} |\rho-\tau|}\frac{1}{\tau^{j+1}} e^{-c_{0} \lambda_{0}''\tau} d\tau
	\leq C'' \frac{e^{-c_{0}\lambda_{0}''\rho}}{\rho^{j+1}},
	\end{align*}
	where $C''$ is a positive constant independent by $j$.\\
	Putting the above considerations together we obtain the (\ref{Int_Inq-2}). 
\end{proof}
\begin{remark}\label{Rk_L-0_Est}
	The functions $h_{0}(\rho)$ and $h_{m-1}(\rho)$ are both solutions of the problem $\Theta_{0} h = 0$.
\end{remark}
\begin{lemma}\label{L-log-Est}
	Let $a$,  $s$ be two positive real numbers and $j\in \mathbb{Z}_{+}$. Then for every $\varepsilon > 0$ there exist a positive constant $C_{\varepsilon}$,
	independent of $j$, such that 
	\begin{align}\label{Int_Est_log}
	\int_{0}^{+\infty} e^{-a\rho\log\rho} \rho^{sj} d\rho \leq C_{\varepsilon} \varepsilon^{j} j^{js}.
	\end{align}
\end{lemma}
\begin{proof}
	We have
	\begin{align*}
	\int_{0}^{+\infty} e^{-a\rho\log\rho} \rho^{sj} d\rho =  \int_{0}^{+\infty} e^{-a\rho\log\rho+c\rho} \rho^{sj} e^{-(c-1)\rho} e^{-\rho} d\rho, 
	\end{align*}
	where $c$ is a positive real number such that $s(2^{-1}\varepsilon)^{-1/s}+1 > c > s\varepsilon^{-1/s}+1$. We have
	\begin{align*}
	\sup_{\rho}\rho^{sj} e^{-(c-1)\rho} =  \left[\left(\frac{s}{c-1}\right)^{s}\right]^{j} e^{-js} j^{sj} \leq \varepsilon^{j} j^{js}.
	\end{align*}
	On the other hand $f(\rho)= -a\rho\log\rho+c\rho$ take its maximum in $\rho_{0} = e^{(c-a)/a}$. 
	We have that $e^{f(\rho_{0})} \leq \exp( a e^{(s(2\varepsilon^{1/s})^{-1}-a+1)a^{-1}}) \doteq C$. Setting $C= C_{\varepsilon}$ we have the assertion. 
\end{proof}
\begin{lemma}\label{Ge_vs_Ga}
	Let $s_{1}$ and $s_{2}$ two positive real number such that $s_{1} < s_{2}$ then $G^{s_{1}}(\Omega) \subset \gamma^{s_{2}}(\Omega)$.
\end{lemma}
\begin{proof}
	We have that $AC^{|\alpha|}|\alpha|^{s|\alpha|}= AC^{|\alpha|} |\alpha|^{(s_{1}-s_{2})|\alpha|}|\alpha|^{s_{2}|\alpha|}$.
	We want to see that
	\begin{align*}
	C^{|\alpha|} |\alpha|^{(s_{1}-s_{2})|\alpha|} \leq C_{\varepsilon}\varepsilon^{|\alpha|} \qquad \forall \, \varepsilon > 0.
	\end{align*}
	We observe that since $s_{2} > s_{1}$ than $C^{|\alpha|} |\alpha|^{(s_{1}-s_{2})|\alpha|}\rightarrow 0$ for $|\alpha | \rightarrow +\infty$.\\
	We take the logarithm of both side of above inequality
	\begin{align*}
	|\alpha| \log\left(\frac{C}{\varepsilon}\right) - (s_{2}-s_{1}) \log |\alpha| \leq \log C_{\varepsilon}.
	\end{align*}
	Consider the function $ f(t)= t \log (C/\varepsilon) - (s_{2}-s_{1})t\log t$ in  $(1, +\infty)$. We have
	$f'(t)= 0$ for $t_{0}= e^{-1} (C/\varepsilon)^{1/(s_{2}-s_{1})}$. Setting $C_{\varepsilon} = \exp(f(t_{0}))$ we have the assertion.
\end{proof}

%
%
%
\bibliographystyle{amsplain}

\end{document}